\documentclass[12pt]{amsart}
\usepackage{txfonts}      %{article} was 12pt latex e
\usepackage{amssymb}
\usepackage{eucal}
\usepackage{amsmath}
\usepackage{amscd}
\usepackage{xcolor}
\usepackage{multicol}
\usepackage[all]{xy}           %xypic macro for latex2.09
\usepackage{graphicx}
\usepackage{color}
\usepackage{colordvi}
\usepackage{xspace}
\usepackage{amssymb}
\usepackage{tikz}
\usepackage{makecell}
\usepackage{appendix}
\usepackage{enumerate,enumitem}
\usepackage{amsthm}
\usepackage[thicklines]{cancel}%划线
\usepackage[compress]{cite}
\usepackage{ifpdf}
\ifpdf
\usepackage[colorlinks,final,backref=page,hyperindex]{hyperref}
\else
\usepackage[colorlinks,final,backref=page,hyperindex,hypertex]{hyperref}
\fi

\usepackage[active]{srcltx} %SRC Specials for DVI Searching
\newtheorem{thm}{Theorem}[section]
\newtheorem{lem}[thm]{Lemma}
\newtheorem{cor}[thm]{Corollary}
\newtheorem{pro}[thm]{Proposition}
\theoremstyle{definition}
\newtheorem{defi}[thm]{Definition}
\newtheorem{ex}[thm]{Example}
\newtheorem{rmk}[thm]{Remark}

\title[ Anti-pre-Poisson bialgebras and relative Rota-Baxter operators] {Anti-pre-Poisson bialgebras and relative Rota-Baxter operators}

\author{Qinxiu Sun}
\address{Department of Mathematics, Zhejiang University of Science and Technology, Hangzhou, 310023} \email{qxsun@126.com}
\author{Min Wu}
\address{Department of Mathematics, Zhejiang University of Science and Technology, Hangzhou, 310023} \email{wuminzju@163.com}


\begin{document}
\begin{abstract}
	
 In this paper, we first introduce the notion of an anti-pre-Poisson bialgebra, which is shown to be equivalent to
  both quadratic anti-pre-Poisson algebras and matched pairs of Poisson algebras. 
  The study of coboundary anti-pre-Poisson bialgebras leads to the anti-pre-Poisson Yang-Baxter equation (APP-YBE).
   Skew-symmetric solutions of this equation give rise to coboundary anti-pre-Poisson bialgebras.
Furthermore, we investigate how solutions without skew-symmetry can also induce such bialgebras, 
prompting the introduction of quasi-triangular and factorizable anti-pre-Poisson bialgebras. In particular,
 solutions of the APP-YBE whose symmetric parts are invariant induce a quasi-triangular anti-pre-Poisson bialgebra. 
 Such solutions are also interpreted as relative Rota-Baxter operators with weights.
Finally, we establish a one-to-one correspondence between quadratic Rota-Baxter anti-pre-Poisson algebras
 and factorizable anti-pre-Poisson bialgebras.
 
\end{abstract}

\maketitle

\vspace{-1.2cm}

\tableofcontents

\vspace{-1.2cm}

\allowdisplaybreaks

\section{Introduction}

Pre-Lie algebras \cite{4} serve as the underlying algebraic structures of symplectic forms on Lie algebras.
 A symmetric analogue of a symplectic form on a Lie algebra is a non-degenerate commutative 2-cocycle \cite{9}. 
 The notion of anti-pre-Lie algebras was introduced in \cite{12} as the underlying algebraic structures of such non-degenerate 
 commutative 2-cocycles on Lie algebras.
These can be viewed as the ''anti-structures'' "anti-structures" of pre-Lie algebras and are characterized as Lie-admissible algebras
 whose negative left multiplication operators yield representations of the commutator Lie algebras.
Anti-pre-Lie algebras are closely connected to various other algebraic structures, including transposed Poisson algebras \cite{5} 
and differential algebras \cite{7}. Builds upon these foundations, Gao, Liu, and Bai introduced anti-dendriform algebras in \cite{10}.
These retain the property of splitting associativity: the sum of their two bilinear operations forms an associative algebra.
However, unlike dendriform algebras, where the left and right multiplication operators 
furnish a representation of the total associative algebra, while in an anti-dendriform algebra, 
it is the negatives of these operators that constitute a representation of the sum associative algebra.
 In particular, for an anti-dendriform algebra $(A, \succ, \prec)$, 
 if the relation $x \succ y = y \prec x$ holds for all $x, y \in A$, then $(A, \ast = \succ)$ is called an anti-Zinbiel algebra \cite{14}.
Notably, there exists a close relationship between anti-dendriform algebras and anti-pre-Lie algebras. 
The study of anti-dendriform algebras and their connections to Novikov-type algebras establishes a foundational framework relevant 
to anti-pre-Lie algebras, as demonstrated in \cite{10}. Specifically, the following commutative diagram holds:
\begin{displaymath}
\xymatrix{
 \text{anti-dendriform algebras} \ar[d] \ar[r] & \text{anti-pre-Lie algebras}  \ar[d] \\
\text{associative algebras} \ar[r] & \text{Lie algebras.}
}
\end{displaymath}
Bai and Liu integrated the structures of anti-Zinbiel algebras and anti-pre-Lie algebras on the same vector space, 
thereby introducing the notion of an anti-pre-Poisson algebra \cite{14}, which bears a close relationship to Poisson algebras.
On the one hand, an anti-pre-Poisson algebra naturally gives rise to a Poisson algebra. 
This is realized through the sub-adjacent commutative associative algebra (from the anti-Zinbiel part)
 and the sub-adjacent Lie algebra (from the anti-pre-Lie part).
On the other hand, an anti-pre-Poisson algebra can also be constructed from an anti-$\mathcal{O}$-operator on a Poisson algebra.

 A bialgebraic structure comprises an algebra and a coalgebra equipped with certain compatibility conditions. 
 In the early 1980s, Drinfeld established the theory of Lie bialgebras \cite{8},
  which were found to have deep connections with the classical Yang-Baxter equation and classical integrable systems. 
  Subsequently, V. Zhelyabin introduced associative D-bialgebras in \cite{24,25}. 
  Aguiar later in \cite{1} studied antisymmetric infinitesimal bialgebras as the associative analogue of Lie bialgebras, 
  which are also equivalent to double constructions of Frobenius algebras \cite{3}. 
  Following the development of infinitesimal bialgebra theory \cite{1,3}, 
  similar bialgebraic frameworks have been extended to various other algebraic structures, 
  such as pre-Lie algebras \cite{2}, Leibniz algebras \cite{22}, 
   Novikov algebras \cite{11} and pre-Novikov algebras \cite{15}.
A Manin triple of Poisson algebras corresponds to a Poisson bialgebra \cite{17},
 offering a natural framework for constructing compatible Poisson brackets in integrable systems.
 However, the approach based on Manin triples with respect to invariant bilinear forms on
  both commutative associative algebras and Lie algebras is not well-suited for transposed Poisson algebras.
Alternatively, Bai and Liu developed bialgebra theories for anti-pre-Lie algebras, transposed Poisson algebras
 and anti-pre-Lie Poisson algebras in \cite{13}, via Manin triples constructed from commutative 2-cocycles on Lie algebras.
We have further established bialgebra theories for anti-dendriform algebras \cite{20} and anti-pre-Novikov algebras \cite{21}.

 Within the theory of Lie bialgebras, coboundary Lie bialgebras especially quasi-triangular Lie bialgebras,
 which play a fundamental role in mathematical physics. As a specialized subclass of quasi-triangular Lie bialgebras,
  factorizable Lie bialgebras provide a crucial link between classical $r$-matrices and certain factorization problems.
   They exhibit diverse applications in integrable systems, see \cite{16} and references therein.
  Recently, these results on factorizable and quasi-triangular structures have been successfully 
   extended to antisymmetric infinitesimal bialgebras \cite{18}, pre-Lie bialgebras \cite{23} and Leibniz bialgebras \cite{6}.
 
It is natural to investigate bialgebra theory for anti-pre-Poisson algebras,
 which serves as the primary motivation for this work. Specifically, we introduce the notion of an anti-pre-Poisson bialgebra. 
 The study of coboundary anti-pre-Poisson bialgebras leads to the anti-pre-Poisson Yang-Baxter equation (APP-YBE),
  where any skew-symmetric solution gives rise to an anti-pre-Poisson bialgebra. More importantly,
   we also examine how solutions without skew-symmetry can induce such bialgebras. In particular,
    we prove that solutions of the APP-YBE whose symmetric parts are invariant yield quasi-triangular anti-pre-Poisson bialgebras.
    Furthermore, we consider factorizable anti-pre-Poisson bialgebras as a special subclass of
     quasi-triangular anti-pre-Poisson bialgebras.
     We show that the double space of any anti-pre-Poisson bialgebra naturally carries a factorizable structure. 
  Finally, we characterize solutions of the APP-YBE with invariant symmetric parts in terms of 
  relative Rota-Baxter operators on anti-pre-Poisson algebras.

 The paper is organized as follows. In Section 2, we review fundamental concepts and results related 
 to anti-pre-Lie algebras and anti-Zinbiel algebras. In particular, we study the representations of anti-pre-Poisson algebras.
  Section 3 builds upon these foundations to develop a bialgebra theory for anti-pre-Poisson algebras. 
  By examining the coboundary case, we introduce the anti-pre-Poisson Yang–Baxter equation (APP-YBE), 
  whose skew-symmetric solutions yield anti-pre-Poisson bialgebras.
  We further define the notion of $\mathcal{O}$-operators on anti-pre-Poisson algebras and use them to construct 
  skew-symmetric solutions of the APP-YBE. In Section 4, we explore quasi-triangular and factorizable anti-pre-Poisson bialgebras.
  We prove that the double of any anti-pre-Poisson bialgebra naturally carries a factorizable structure. 
  Finally, in Section 5, we introduce quadratic Rota-Baxter anti-pre-Poisson algebras and 
  establish their correspondence with factorizable anti-pre-Poisson bialgebras.
 
 {\bf Notations.} Throughout the paper, $k$ is a field.  All vector spaces and algebras are over $k$. 
 All algebras are finite-dimensional, although many results still hold in the infinite-dimensional case.
 Let $V$ be a vector space with a binary operation $\ast$. Define linear maps
$L_{\ast}, R_{\ast},\mathrm{ad}:V\rightarrow \hbox{End}(V)$ by
 $L_{\ast}(a)b:=a\ast b, \  \ R_{\ast}(a)b:=b\ast a, \ \ \mathrm{ad}(a)b=a\ast b-b\ast a$~ for all$~a, b\in V$.
Assume that
 $r=\sum\limits_{i}a_i\otimes b_i \in V\otimes V$. Put
 \begin{small}
\begin{align*}
r_{12}\ast r_{13}:=\sum_{i,j}a_i\ast a_j\otimes b_i\otimes b_j,\;r_{23}\ast r_{12}:=\sum_{i,j}a_j\otimes a_i\ast b_j\otimes b_i,\;
r_{31}\ast r_{23}:=\sum_{i,j}b_i\otimes a_j\otimes a_i\ast b_j,\\
r_{21}\ast r_{13}:=\sum_{i,j}b_i\ast a_j\otimes a_i\otimes b_j,\;
r_{32}\ast r_{21}:=\sum_{i,j}b_j\otimes b_i\ast a_j\otimes a_i,\;
r_{31}\ast r_{32}:=\sum_{i,j}b_i\otimes b_j\otimes a_i\ast a_j,\\
r_{13}\ast r_{32}:=\sum_{i,j}a_i\otimes b_j\otimes b_i\ast a_j,\;
r_{23}\ast r_{21}:=\sum_{i,j}b_j\otimes a_i\ast a_j\otimes b_i,\;
r_{21}\ast r_{31}:=\sum_{i,j}b_i\ast b_j\otimes a_i\otimes a_j,\\
r_{23}\ast r_{13}:=\sum_{i,j}a_i \otimes a_j \otimes b_i\ast b_j,\;
r_{12}\ast r_{31}:=\sum_{i,j}a_i\ast b_j\otimes b_i\otimes a_j.
\end{align*}\end{small}

\section{Representation of anti-pre-Poisson algebras}
This section commences with a review of the foundational properties of anti-Zinbiel 
and anti-pre-Lie algebras, laid out in \cite{10,12,13,20}.

An {\bf anti-dendriform algebra} is a vector space $A$ together with two bilinear maps $\succ,\prec:A\times A\longrightarrow A$ such that
 \begin{align*}&x\succ (y\succ z)=-(x\cdot y)\succ z=-x\prec (y\cdot z)=(x\prec y)\prec z,\\&
(x\succ y)\prec z=x\succ (y\prec z),\end{align*} 
  for all $x,y,z\in A$, where $x\cdot y= x\succ y+x\prec y$. $(A,\cdot)$ is an associative algebra, which is called the associated associative 
  algebra of the anti-dendriform algebra $(A,\succ,\prec)$. Furthermore, $(A,\succ,\prec)$ is called the compatible 
  anti-dendriform algebra structure on $(A,\cdot)$. When $x\succ y=y\prec x=x\ast y$, $(A,\ast)$ is called an anti-Zinbiel algebra. 
  More precisely, an anti-Zinbiel algebra is a vector space $A$ together with a bilinear map $\ast:A\times A\longrightarrow A$ such that
\begin{align*}&x\ast (y\ast z)=-(x\ast y+y\ast x)\ast z=- (y\ast z+z\ast y)\ast x=z\ast(y\ast x),
\\& z\ast(x\ast y)=x\ast(z\ast y),\end{align*}  for all $x,y,z\in A$.
 Denote $x\star y= x\ast y+y\ast x$, $(A,\star)$ is a commutative associative algebra, which is called the associated associative 
  algebra of the anti-Zinbiel algebra $(A,\ast)$. Furthermore, $(A,\ast)$ is called the compatible 
  anti-Zinbiel algebra structure on $(A,\star)$.

 Let $(A, \cdot)$ be a commutative associative algebra and $\omega$ be a bilinear form on $(A, \cdot)$. If $\omega$ is
symmetric and satisfies 
\begin{equation}\omega(x \cdot y, z) + \omega(y \cdot z, x)+ \omega(z \cdot x, y)=0, \forall~x, y, z \in A,\end{equation}
 then $\omega$ is called a commutative Connes cocycle \cite{10}.
 
 Let $(A,\ast)$ be an anti-Zinbiel algebra. A bilinear form $\omega$ on $(A,\ast)$ is
called invariant if
\begin{equation} \label{C1} \omega(x \ast y, z)=-\omega(x, y\star z), ~\forall~x, y, z \in A.\end{equation}

\begin{pro} \cite{10}\label{Qa1}
Let $(A,\ast)$ be an anti-Zinbiel algebra and $\omega$ be a symmetric invariant
bilinear form on $(A,\ast)$. Then $\omega$ is a commutative Connes cocycle on the associated commutative associative
algebra $(A,\star)$. Conversely, assume that $(A,\star)$ is a commutative associative algebra and $\omega$ is a non-degenerate
commutative Connes cocycle on $(A,\star)$. Then 
$\omega$ is invariant on the compatible anti-Zinbiel
algebra $(A,\ast)$ defined by Eq.~ (\ref{C1}).
\end{pro}

 A {\bf representation (bimodule)} of an anti-Zinbiel algebra $(A,\ast)$ is a triple
  $(V,l_{\ast}, r_{\ast})$, 
 where $V$ is a
vector space and $l_{\ast}, r_{\ast}: A \longrightarrow \hbox{End}
(V)$ are two linear maps satisfying the following relations, for all
$x,y\in A$,
 \begin{align*}&l_{\ast}(x)l_{\ast}(y)=-l_{\ast}(x\star y)=-r_{\ast}(x)l_{\star}(y)=r_{\ast}(y\ast x),\\&
l_{\ast}(x)r_{\ast}(y)=-r_{\ast}(y)l_{\star}(x)=-r_{\ast}(x)l_{\star}(y)=l_{\ast}(y)r_{\ast} (x),\\&
r_{\ast}(x\ast y)=l_{\ast}(x)r_{\ast}(y), \ \ \ 
l_{\ast}(y)l_{\ast}(x)=l_{\ast}(x)l_{\ast} (y),\end{align*} 
where $l_{\star}(x)=l_{\ast}(x)+r_{\ast}(x)$ and $x\star y=x\ast y+y\ast x$.

Let $A$ and $V$ be vector spaces. For
a linear map $f: A \longrightarrow \hbox{End} (V)$, define a linear
map $f^{*}: A \longrightarrow \hbox{End} (V^{*})$ by $\langle
f^{*}(x)u^{*},v\rangle=-\langle u^{*},f(x)v\rangle$ for all $x\in A,
u^{*}\in V^{*}, v\in V$, where $\langle \ , \ \rangle$ is the usual
pairing between $V$ and $V^{*}$.

\begin{pro} \cite{20} \label{zr} Let $(V,l_{\ast}, r_{\ast})$ be a representation of an anti-Zinbiel algebra $(A,\ast)$.
 Then
\begin{enumerate}
	\item $(V,-l_{\ast})$ is a representation of the associated
 associative algebra $(A,\star)$.
 \item $(V,l_{\ast}+r_{\ast})$ is a representation of the associated
 associative algebra $(A,\star)$.
	\item $(V^{*},l_{\ast}^{*}+r_{\ast}^{*},-r_{\ast}^{*})$ is also a
	representation of $(A,\ast)$. We call it the {\bf dual representation}.
	\item $(V^{*},l_{\ast}^{*})$ is a
	representation of $(A,\star)$.
\item $(V^{*},-l_{\ast}^{*}-r_{\ast}^{*})$ is a
	representation of $(A,\star)$. 
\end{enumerate}
\end{pro}

An {\bf anti-pre-Lie algebra} is a vector space $A$ together with a bilinear map $\circ:A \times A \longrightarrow A$ such that
	\begin{align*}
		&x\circ(y\circ z)-y\circ(x\circ z)=[y,x]\circ z,\\&
[x,y]\circ z+[y,z]\circ x+[z,x]\circ y=0,
	\end{align*} 
 for all $x,y,z\in A$, where $[x, y]=x\circ y -y \circ x$. $(A, [ \ , \ ])$ is a Lie algebra, which is called the {\bf sub-adjacent Lie algebra} 
of $(A,\circ)$ and is denoted by $\frak g(A)$. Moreover, $(A,\circ)$ is called a {\bf compatible anti-pre-Lie algebra} of $(A, [ \ , \ ])$.

Recall that a commutative 2-cocycle \cite{9} on a Lie algebra $(g, [ \ , \ ])$ is
a symmetric bilinear form $\omega$ such that
\begin{equation} \omega([x, y], z) + \omega([y, z], x) + \omega([z, x], y) = 0, ~\forall~x, y, z\in g.\end{equation}

Let $(A, \circ)$ be an anti-pre-Lie algebra. A bilinear form $\omega$ on $(A,\circ)$ is
called invariant if
\begin{equation} \label{C2}
	\omega(x\circ y, z)=\omega(y,[x, z]), ~~\forall~x, y, z\in g. 
\end{equation}

\begin{pro} \cite{12}\label{Qa2}
Let $(A,\circ)$ be an anti-pre-Lie algebra and $\omega$ be a symmetric invariant
bilinear form on $(A,\circ)$. Then $\omega$ is a commutative 2-cocycle on the associated Lie
algebra $(A,[ \ , \ ])$. Conversely, suppose that $(A,[ \ , \ ])$ is a Lie algebra and $\omega$ is a non-degenerate
commutative 2-cocycle on $(A,[ \ , \ ])$ . Then $\omega$ is invariant on the compatible anti-pre-Lie
algebra $(A,\circ)$ defined by Eq~(\ref{C2}).
\end{pro}

 A representation of an anti-pre-Lie algebra $(A, \circ)$ is a triple
$(V,l_{\circ},r_{\circ} )$, such that $V$ is a vector space and $l_{\circ}, r_{\circ}: A \longrightarrow \hbox{End}(V)$ 
 are linear maps for all $x, y \in A$ satisfying:
	\begin{align*}&l_{\circ}(y \circ x)-l_{\circ}(x \circ y)=l_{\circ} (x)l_{\circ}(y)-l_{\circ}(y)l_{\circ}(x), \\&
r_{\circ}(x \circ y) = l_{\circ}(x)r_{\circ}(y) + r_{\circ}(y)l_{\circ}(x) -r_{\circ}(y)r_{\circ}(x), \\&
l_{\circ}(y \circ x)-l_{\circ}(x \circ y) = r_{\circ}(x)l_{\circ}(y)-r_{\circ}(y)l_{\circ}(x) -r_{\circ}(x)r_{\circ}(y)+r_{\circ}(y)r_{\circ}(x).\end{align*}

\begin{pro}\cite{12}\label{pr} Let $(V,l_{\circ},r_{\circ} )$ be a representation of an anti-pre-Lie algebra $(A, \circ)$.
\begin{enumerate}
\item $(V,-l_{\circ})$ and $(V,l_{\circ}-r_{\circ})$ are representations of the sub-adjacent Lie algebra $(\frak g(A), [ \ , \ ])$.
\item $(V^{*},r_{\circ}^{*}-l_{\circ}^{*},r^{*}_{\circ})$ is a representation of $(A,\circ)$.
\item $(V^{*},-l_{\circ}^{*})$ is a representation of the sub-adjacent Lie algebra $(\frak g(A), [ \ , \ ])$.
\end{enumerate}
\end{pro}

Now we come to study the representation theory of anti-pre-Poisson algebras.

\begin{defi} \cite{1}
	An {\bf anti-pre-Poisson algebra} is a triple $(A,
\ast,\circ )$, where $(A, \ast )$ is an anti-Zinbiel algebra,
$(A,\circ )$ is an anti-pre-Lie algebra and the following conditions hold:
\begin{align}&\label{ppa eq1.1}(x\circ y-y \circ x)\ast z=y\ast( x\circ
z)-x\circ (y\ast z),\\&\label{ppa eq1.2}(x\ast y+y \ast x)\circ z=-x\ast (y\circ z)-y\ast( x\circ
z),
\\&\label{ppa eq1.3}z\circ (x\ast y+y \ast x)+x\ast (y\circ z)+y\ast( x\circ
z)=x\ast (z\circ y)+y\ast (z\circ x)
,\end{align}for all $x,y,z\in A$.\end{defi}

By Eq.~(\ref{ppa eq1.1}), we obtain
\begin{align}\label{ppa eq1.0}y\ast( x\circ
z)+x\ast (y\circ z)=x\circ (y\ast z)+y\circ( x\ast z).\end{align}

Let $(A, \ast,\circ )$ be an anti-pre-Poisson algebra. Define
	$$x\star y=x\ast y+y\ast x,~[x,y]=x\circ y-y\circ x,~\forall~x,y\in A.$$ Then
	$(A, \star,[ \ , \ ] )$ is a Poisson algebra, that is, $(A,\star)$ is a commutative associative algebra, 
$(A,[ \ , \ ])$ is  a Lie algebra and the following equation holds:
	\begin{align}\label{ppa eq1.01}
		[z,x\star y]=[z,x]\star y+x\star [z,y],\;\forall x,y,z\in A.
	\end{align}
	Moreover, $(A, \star,[ \ , \ ] )$ is called the {\bf sub-adjacent Poisson algebra} of $(A,\ast,\circ)$, 
and $(A,\ast,\circ)$ is called the {\bf compatible anti-pre-Poisson algebra} of $(A, \star,[ \ , \ ] )$.

\begin{ex} \label{Ae} Let $A$ be a vector space with a basis $\{e_1, e_2, e_3\}$. Define two bilinear
	maps $\ast,\circ:A \times A \longrightarrow A$ respectively by (only non-zero multiplications are listed):
	\begin{equation*}
		e_1\ast e_1=e_3,\ \ \ e_1\circ e_2=e_3. 
	\end{equation*}
 By direct calculation, $(A,\ast,\circ)$ is an anti-pre-Poisson algebra.
\end{ex}
	
	\begin{defi} 
		Let $(A,\ast,\circ)$ be an anti-pre-Poisson algebra,
		$V$ a vector space and $l_{\ast},r_{\ast},l_{\circ},r_{\circ}: A
		\longrightarrow \hbox{End} (V)$ be linear maps. Then
		$(V,l_{\ast},r_{\ast},l_{\circ},r_{\circ})$ is called a
		{\bf representation} of $(A,\ast,\circ)$ if $(V,l_{\ast},r_{\ast})$ is
		a representation of $(A,\ast)$, $(V,l_{\circ},r_{\circ})$ is a representation of
		$(A,\circ)$ and they satisfy the following compatible
		conditions for all $x,y\in A$:
		\begin{align}&\label{r1}l_{\ast}(x\circ y-y\circ x)=l_{\ast}(y)l_{\circ}(x)-l_{\circ}(x)l_{\ast}(y),\\&
		\label{r2}l_{\circ}(x\ast y+y\ast x)=-l_{\ast}(x)l_{\circ}(y)-l_{\ast}(y)l_{\circ}(x),\\&
		\label{r3}r_{\ast}(x)(l_{\circ}(y)-r_{\circ}(y))=r_{\circ}(y\ast x)-l_{\ast}(y)r_{\circ}(x),\\&
		\label{r4}r_{\circ}(x)(l_{\ast}(y)+r_{\ast}(y))=-r_{\ast}(y\circ x)-l_{\ast}(y)r_{\circ}(x),
	\\&\label{r5}r_{\ast}(x)(l_{\circ}(y)-r_{\circ}(y))=r_{\ast}(y\circ x)-l_{\circ}(y)r_{\ast}(x),
\\&\label{r6}r_{\circ}(x\ast y+y\ast x)+l_{\ast}(x)l_{\circ}(y)+l_{\ast}(y)l_{\circ}(x)=l_{\ast}(x)r_{\circ}(y)+l_{\ast}(y)r_{\circ}(x),
\\&\label{r7}l_{\circ}(x)(l_{\ast}(y)+r_{\ast}(y))+r_{\ast}(y\circ x)+l_{\ast}(y)r_{\circ}(x)=l_{\ast}(y)l_{\circ}(x)+r_{\ast}(x\circ y).\end{align}
	 \end{defi}
By Eq.~(\ref{r7}), we obtain
\begin{equation}\label{r8}
l_{\ast}(x)(l_{\circ}(y)-r_{\circ}(y))-l_{\circ}(y)(l_{\ast}(x)+r_{\ast}(x))
=r_{\ast}(x\circ y-y\circ x).\end{equation}

By Eqs.~(\ref{r1}) and (\ref{r2}), we obtain
\begin{equation}\label{r6}l_{\circ}(x\ast y+y\ast
x)=-l_{\circ}(y)l_{\ast}(x)-l_{\circ}(x)l_{\ast}(y).\end{equation}

By Eqs.~(\ref{r3}) and (\ref{r5}), we get
\begin{equation}\label{r9}r_{\ast}(y\circ x)-r_{\circ}(y\ast
x)= l_{\circ}(y)r_{\ast}(x)-l_{\ast}(y)r_{\circ}(x).\end{equation}

\begin{pro} \label{Dr} Let $(A,\ast,\circ)$ be an anti-pre-Poisson algebra
 and $(V,l_{\ast},r_{\ast} ,l_{\circ},r_{\circ})$ be a representation of
$(A,\ast,\circ)$. Then
\begin{enumerate}
	\item $(V^{*},l_{\ast}^{*}+r_{\ast}^{*},-r_{\ast}^{*},r_{\circ}^{*}-l_{\circ}^{*},r^{*}_{\circ})$
	is also a representation of $(A,\ast,\circ)$. We call it the {\bf dual
		representation} .
	\item $(V,-l_{\ast} ,-l_{\circ} )$ 
	is a representation of $(A,\star,[ \ ,\ ])$.
\item $(V,l_{\ast}+r_{\ast} ,l_{\circ}-r_{\circ} )$ 
	is a representation of $(A,\star,[ \ ,\ ])$.
\item $(V,l_{\ast}^{*} ,-l_{\circ}^{*} )$ 
	is a representation of $(A,\star,[ \ ,\ ])$.
\end{enumerate}
\end{pro}

\begin{proof} (a) On the basis of Propositions
\ref{zr} and \ref{pr}, $(V^{*},l_{\ast}^{*}+r_{\ast}^{*},-r_{\ast}^{*})$ is a
representation of $(A,\ast)$ and
$(V^{*},r_{\circ}^{*}-l_{\circ}^{*},r^{*}_{\circ})$ is  a
representation of $(A,\circ)$. We only need to check that Eqs.~(\ref{r1})-(\ref{r7}) hold for
$(V^{*},l_{\ast}^{*}+r_{\ast}^{*},-r_{\ast}^{*},r_{\circ}^{*}-l_{\circ}^{*},r^{*}_{\circ})$.
As an example, we give an explicit proof of (\ref{r1}). By Eqs.~(\ref{r1}), (\ref{r4}) and (\ref{r5}),
we have for any $x,y\in A,u\in V,v^{*}\in V^{*} $,
\begin{align*}&\langle [(l_{\ast}^{*}(x)+r_{\ast}^{*})(x\circ y-y\circ x)-(l_{\ast}^{*}(x)+r_{\ast}^{*})(y)(r_{\circ}^{*}-l_{\circ}^{*})(x)
+(r_{\circ}^{*}-l_{\circ}^{*})(x)(l_{\ast}^{*}(x)+r_{\ast}^{*})(y)]
v^{*},u\rangle\\=&
\langle v^{*},(l_{\ast}+r_{\ast})(y\circ x-x\circ y)u+(r_{\circ}-l_{\circ})(x)(l_{\ast}+r_{\ast})(y)u
-(l_{\ast}+r_{\ast})(y)(r_{\circ}-l_{\circ})(x)u
\rangle
\\=&0,\end{align*} 
which yields that Eq.~(\ref{r1}) holds for substituting $(l_{\ast},l_{\circ})$ by 
$(l_{\ast}^{*}+r_{\ast}^{*},r_{\circ}^{*}-l_{\circ}^{*})$.
Items (b) and (c) can be proved similarly. Combining Items (a) and (c), we get that Item (d) holds.
\end{proof}

\begin{ex} Let $(A,\ast,\circ)$ be an anti-pre-Poisson algebra.
Then $(A,L_{\ast}, R_{\ast},L_{\circ}, R_{\circ})$ is a
representation of $(A,\ast,\circ)$, which is called the {\bf regular
representation} of $(A,\ast,\circ)$. Moreover, 
$(A^{*},L_{\ast}^{*}+R_{\ast}^{*},-R_{\ast}^{*},R_{\circ}^{*}-L_{\circ}^{*},R^{*}_{\circ})$ is the dual
representation.
Furthermore, $(A,  -{L} _{\ast},  -{L} _{\circ})$ and 
 $(A^{*},{L}^{*}_{\ast},  -{L}^{*}_{\circ})$ are representations of $(A,\star,[ \ , \ ])$.
\end{ex}

Recall the matched pairs of Poisson algebras from \cite{17}.

\begin{pro} \cite{17}\label{Df} Let $(P_{1},\star_{1},[ \ , \ ]_1)$ and
$(P_{2},\star_{2},[ \ , \ ]_2)$ be two Poisson
algebras. Suppose that there exist four linear maps
$\mu_{1},\rho_1:P_1\longrightarrow \hbox{End}(P_2)$
and $\mu_{2},\rho_2:P_2\longrightarrow
\hbox{End}(P_1)$.
Define multiplications on $P_{1}\oplus P_{2}$ by
\begin{equation}\label{eq:140}[x+a,y+b]=[x,y]_{1}+\rho_{2}(a)y-\rho_{2}(b)x+[a,b]_{2}
	+\rho_{1}(x)b-\rho_{1}(y)a,\end{equation}
\begin{equation}\label{eq:141}(x+a)\star
	(y+b)=x\star_{1}y+\mu_{2}(a)y+\mu_{2}(b)x+a\star_{2}b+\mu_{1}(x)b+\mu_{1}(y)a,\;\forall x,y\in P_{1},a,b\in P_{2}.\end{equation}
Then $(P_{1}\oplus P_{2},\star,[\ ,\ ])$ is a Poisson algebra if and only if the following conditions hold:
\begin{enumerate}
	\item $(P_{1},P_{2},\mu_{1},\mu_{2})$
	is a matched pair of commutative associative algebras.
	\item $(P_{1},P_{2},\rho_{1},\rho_{2})$
	is a matched pair of Lie algebras.
	\item $(P_{2},\mu_{1},\rho_{1})$ is a 
	representation of  $(P_1,\star_{1},[ \ , \ ]_1)$.
	\item $(P_{1},\mu_{2},\rho_{2})$ is a 
	representation of  $(P_2,\star_{2},[ \ , \ ]_2)$.
\item The following compatible conditions hold:
\begin{small}
\begin{equation}\label{P1}
	\rho_{2}(a)(x\star_{1}y)=(\rho_{2}(a)x)\star_{1}y
	+x\star_1(\rho_{2}(a)y)-\mu_{2}(\rho_{1}(x)a)y
	-\mu_{2}(\rho_{1}(y)a)x, \end{equation}
\begin{equation}\label{P2}[x,\mu_{2}(a)y]_{1}-\rho_{2}(\mu_{1}(y)a)x
	=\mu_{2}(\rho_{1}(x)a)y-(\rho_{2}(a)x)\star_{1}y
	+\mu_{2}(a)([x,y]_{1}),\end{equation}
\begin{equation}\label{P3}\rho_{1}(x)(a\star_{2}b)=(\rho_{1}(x)a)\star_{2}b
	+a\star_2(\rho_{1}(x)b)-\mu_{1}(\rho_{2}(a)x)b
	-\mu_{1}(\rho_{2}(b)x)a,\end{equation}
\begin{equation}\label{P4}[a,\mu_{1}(x)b]_{2}-\rho_{1}(\mu_{2}(b)x)a
	=\mu_{1}(\rho_{2}(a)x)b-(\rho_{1}(x)a)\star_{2}b
	+\mu_{1}(x)([a,b]_{2}),
\end{equation}
\end{small}
\end{enumerate} 
for any $x,y\in P_{1}$ and $a,b\in P_{2}$.
In this case, 
we denote this Poisson algebra simply by $P_{1}\bowtie P_{2}$ and 
$(P_{1},P_{2},\mu_{1},\rho_{1},\mu_{2},\rho_{2})$
satisfying the above conditions is called a {\bf matched pair of
Poisson algebras}.\end{pro}

\begin{pro}\label{Mp}
	 Let $(A_1,\ast_1,\circ_1)$ and $(A_2,\ast_2,\circ_2)$ be two anti-pre-Poisson
algebras. Suppose that there are linear maps
$l_{1},r_{1},\rho_1,\mu_1:A_1\longrightarrow
\hbox{End} (A_2)$ and
$l_{2},r_{2},\rho_2,\mu_2:A_2\longrightarrow
\hbox{End} (A_1)$.
Define multiplications on $A_1\oplus A_2$   by
\begin{eqnarray}
&&(x+a)\ast
(y+b)=x{\ast_1}y+l_{2}(a)y+r_{2}(b)x+a{\ast_2}b+l_{1}(x)b+r_{1}(y)a
,\label{ppmp eq1.1}\\
&&(x+a)\circ
(y+b)=x{\circ_1}y+\rho_2(a)y+\mu_2(b)x+a{\circ_2}b+\rho_1(x)b+\mu_1(y)a
,~~\forall~x,y\in A_1,a,b\in A_2.\label{ppmp eq1.2}
\end{eqnarray}
Then $(A_{1}\oplus A_{2},\ast,\circ)$ is an anti-pre-Poisson algebra if and only if  the following conditions hold:
\begin{enumerate}
	\item $(A_2,l_{1},r_{1},\rho_1,\mu_1)$ is a representation of $(A_1,\ast_1,\circ_1)$.
	\item $(A_1,l_{2},r_{2},\rho_2,\mu_2)$ is a representation of $(A_2,\ast_2,\circ_2)$.
	\item $(A_1,A_2, l_{1},r_{1},l_{2},r_{2})$ is a matched pair of anti-Zinibel algebras.
	\item $(A_1,A_2,\rho_1,\mu_1,\rho_2,\mu_2)$ is a matched pair of anti-pre-Lie algebras.
	\item The following compatible conditions hold:
\begin{small}
	\begin{align}&\label{ppmp eq1.4}r_{2}(a)[x, y]_1=-x\circ_{1}
		r_{2}(a)y+y\ast_{1}\mu_2(a)x-\mu_2(l_{1}(y)a)x+r_{2}(\rho_1(x)a)y,
\\&\label{ppmp eq1.5}(\mu_2(a)x-\rho_2(a)x)
		\ast_{1} y+l_{2}(\rho_1(x)a-\mu_1(x)a)y=
		l_{2}(a)(x\circ_1 y)-x\circ_{1}
		l_{2}(a)y-\mu_2(r_{1}(y)a)x,
\\&\label{ppmp eq1.6}(\rho_2(a)x-\mu_2(a)x)
		\ast_{1} y+l_{2}(\mu_1(x)a-\rho_1(x)a)y=
		x\ast_{1}
		\rho_2(a)y+r_{2}(\mu_1(y)a)x-\rho_2(a)(x\ast_{1}y),
\\&\label{ppmp eq1.7}\mu_2(a)(x\star_{1} )=-x\ast_{1}
		\mu_2(a)y-y\ast_{1}
		\mu_2(a)x-r_{2}(\rho_1(y)a)x-r_{2}(\rho_1(x)a)y,
\\&\label{ppmp eq1.8}(l_{\star_2}(a)x)
		\circ_{1} y+\rho_2(l_{\star_1}(x)a)y=-x\ast_{1}
		\rho_2(a)y-r_{2}(\mu_1(y)a)x-
		l_{2}(a)(x\circ_1 y),
\\&\label{ppmp eq1.9}r_{1}(x)[a,b]_{2} =r_{1}(\rho_2(a)x)b-a\circ_{2}
		r_{1}(x)b+b\ast_{2}
		\mu_1(x)a-\mu_1(l_{2}(b)x)a,
\\&\label{ppmp eq1.10}(\mu_1(x)a-\rho_1(x)a)
		\ast_{2} b+l_{1}(\rho_2(a)x-\mu_2(a)x)b=l_{1}(x)(a\circ_2 b)-a\circ_{2}
		l_{1}(x)b-\mu_1(r_{2}(b)x)a,
\\&\label{ppmp eq1.11}(\rho_1(x)a-\mu_1(x)a)
		\ast_{2} b+l_{1}(\mu_2(a)x-\rho_2(a)x)b=
		-\rho_1(x)(a\ast_{2}b)+a\ast_{2}
		\rho_1(x)b+r_{1}(\mu_2(b)x)a,
\\&\label{ppmp eq1.12}\mu_1(x)(a\star_{2} b)+a\ast_{2}
		\mu_1(x)b+b\ast_{2}
		\mu_1(x)a+r_{1}(\rho_2(b)x)a+r_{1}(\rho_2(a)x)b=0,
\\&\label{ppmp eq1.13}(l_{\star_1}(x)a)
		\circ_{2} b+\rho_1(l_{\star_2}(a)x)b+a\ast_{2}
		\rho_1(x)b+r_{1}(\mu_2(b)x)a+
		l_{1}(x)(a\circ_2 b)=0,
\\&\label{ppmp eq1.14}(\rho_{2}-\mu_{2})(a)(x\star_{1} y)-x\ast_1\rho_2(a)y-y\ast_1\rho_2(a)x-r_2(\mu_1(y)a)x-r_2(\mu_1(x)a)y=0,
\\&\label{ppmp eq1.15}[y,l_{\star_2}(a)x]_1+(\mu_2-\rho_2)(l_{\star_1}(x)a)y
-l_2(a)(y\circ_1 x)-x\ast_1\mu_2(a)y-r_2(\rho_1(y)a)x=0,
\\&\label{ppmp eq1.16}[b,l_{\star_1}(x)a]_2+(\mu_1-\rho_1)(l_{\star_2}(a)x)b
-l_1(x)(b\circ_2 a)-a\ast_2\mu_1(x)b-r_1(\rho_2(b)x)a=0,
\\&\label{ppmp eq1.17}(\rho_{1}-\mu_{1})(x)(a\star_{2} b)-a\ast_{2} \rho_1(x)b-b\ast_{2} \rho_1(x)a
-r_1(\mu_2(b)x)a-r_1(\mu_2(a)x)b=0,\end{align}
\end{small}
\end{enumerate}
where $l_{\star_1}=l_1+r_1,~l_{\star_2}=l_2+r_2,~[x,y]_1=x\circ_1 y-y\circ_1 x,~[a,b]_2=a\circ_2 b-b\circ_2 a$
and $x\star_1 y=x\ast_1y+y\ast_1 x,~a\star_2 b=a\ast_2 b+b\ast_2 a$ for all $x,y\in A_1,a,b\in A_2$.
In this case, we
denote this anti-pre-Poisson algebra by $A_1\bowtie A_2$, and 
$(A_1, A_2,
l_{1},r_{1},\rho_1,\mu_1,l_{2},r_{2},\rho_2$,
$\mu_2)$
satisfying the above conditions is called a {\bf matched pair of
anti-pre-Poisson algebras}. \end{pro}

\begin{proof} To prove that $(A_1\oplus A_2,\ast,\circ)$ is an anti-pre-Poisson algebra,
		we only need to check that Eqs.~(\ref{ppa eq1.1})-(\ref{ppa eq1.3}) hold for 
all $x+a,y+b,z+c\in A_1\oplus A_2$ with $x,y,z\in A_1,a,b,c\in A_2$. 
Indeed, we may verify the following cases through direct computations:
\begin{enumerate}
\item Eq.~(\ref{ppa eq1.1}) holds for $(x,y,a)$ if and only if Eqs.~(\ref {r1}) and (\ref {ppmp eq1.4}) hold.
\item Eq.~(\ref{ppa eq1.1}) holds for $(x,a,y)$ if and only if Eqs.~(\ref {r5}) and (\ref {ppmp eq1.5}) hold.
\item Eq.~(\ref{ppa eq1.1}) holds for $(a,x,y)$ if and only if Eqs.~(\ref {r3}) and (\ref {ppmp eq1.6}) hold.
\item Eq.~(\ref{ppa eq1.1}) holds for $(x,a,b)$ if and only if Eqs.~(\ref {r3}) and (\ref {ppmp eq1.11}) hold.
\item Eq.~(\ref{ppa eq1.1}) holds for $(a,x,b)$ if and only if Eqs.~(\ref {r5}) and (\ref {ppmp eq1.10}) hold.
\item Eq.~(\ref{ppa eq1.1}) holds for $(a,b,x)$ if and only if Eqs.~(\ref {r1}) and (\ref {ppmp eq1.9}) hold.
\item Eq.~(\ref{ppa eq1.2}) holds for $(x,y,a)$ if and only if Eqs.~(\ref {r2}) and (\ref {ppmp eq1.8}) hold.
\item Eq.~(\ref{ppa eq1.2}) holds for $(x,a,y)$ if and only if Eqs.~(\ref {r4}) and (\ref {ppmp eq1.8}) hold.	
\item Eq.~(\ref{ppa eq1.2}) holds for $(x,a,b)$ if and only if Eqs.~(\ref {r4}) and (\ref {ppmp eq1.13}) hold.
\item Eq.~(\ref{ppa eq1.2}) holds for $(a,b,x)$ if and only if Eqs.~(\ref {r2}) and (\ref {ppmp eq1.12}) hold.
\item Eq.~(\ref{ppa eq1.3}) holds for $(x,y,a)$ if and only if Eqs.~(\ref {r6}) and (\ref {ppmp eq1.14}) hold.
\item Eq.~(\ref{ppa eq1.3}) holds for $(a,x,y)$ if and only if Eqs.~(\ref {r7}) and (\ref {ppmp eq1.15}) hold.	
\item Eq.~(\ref{ppa eq1.3}) holds for $(x,a,b)$ if and only if Eqs.~(\ref {r7}) and (\ref {ppmp eq1.16}) hold.
\item Eq.~(\ref{ppa eq1.3}) holds for $(a,b,x)$ if and only if Eqs.~(\ref {r6}) and (\ref {ppmp eq1.17}) hold.
\end{enumerate}
\end{proof}

\begin{pro}\label{Mb2}
	Let $(A_{1}, \ast_1,\circ_1)$ and  $(A_{2}, \ast_2,\circ_2)$ be anti-pre-Poisson algebras, 
whose sub-adjacent Poisson algebras are $(A_{1},\star_{1},[ \ , \ ]_{1})$ and $(A_{2},\star_{2},[ \ , \ ]_{2})$.
	If $(A_1, A_2,
	l_{\ast_1},r_{\ast_1},l_{\circ_1},r_{\circ_1},l_{\ast_2},r_{\ast_2}, $ \ \ \ \ $l_{\circ_2},r_{\circ_2})$
	 is a matched pair of
		anti-pre-Poisson algebras. Then $(A_1, A_2,
		l_{\ast_1}+r_{\ast_1},l_{\circ_1}-r_{\circ_1},l_{\ast_2}+r_{\ast_2},l_{\circ_2}-r_{\circ_2})$ is a matched pair of Poisson algebras.
\end{pro}
\begin{proof}
	By Proposition \ref{Mp}, there is an anti-pre-Poisson algebra on $A_{1}\oplus A_{2}$ given
 by \eqref{ppmp eq1.1} and \eqref{ppmp eq1.2}. The sub-adjacent Poisson algebra is given by 
	{\small
	\begin{eqnarray*}
(x+a)\star(y+b)&=&(x+a)\ast (y+b)+ (y+b)\ast (x+a)\\
&=&x\star_{1}y+	(l_{2}+r_{2})(a)y+(l_{2}+r_{2})(b)x+
a\star_{2}b+	(l_{1}+r_{1})(x)b+(l_{1}+r_{1})(y)a,\\
\ [(x+a),(y+b)]&=&(x+a)\circ (y+b)- (y+b)\circ (x+a)\\
&=&[x,y]_{1}+(\rho_2-\mu_2)(a)y-(\rho_2-\mu_2)(b)x+
[a,b]_{2}+(\rho_1-\mu_1)(x)b-(\rho_1-\mu_1)(y)a.
	\end{eqnarray*}}By Proposition \ref{Df}, $(A_1, A_2,
		l_{\ast_1}+r_{\ast_1},l_{\circ_1}-r_{\circ_1},l_{\ast_2}+r_{\ast_2},l_{\circ_2}-r_{\circ_2})$ is a matched pair of Poisson algebras.
\end{proof}
%%%%%%%%%%%%%%%%%%%%%%%%%%%%%%%%%%%%%%%%%%%%%%%%%%%%%%%%%%%%%%%%%%%%%%%%%%%%%%%%%%%%%%%%%%%%%%%%%%%%%%%
\section{Anti-pre-Poisson bialgebras }

\subsection{Anti-Zinbiel bialgebras and anti-pre-Lie bialgebras}
We begin with recalling the bialgebra theories for anti-Zinbiel algebras and anti-pre-Lie bialgebras \cite{13}.
Since an anti-Zinbiel algebra is a special anti-dendriform algebra, 
one can turn to \cite{20} for the anti-Zinbiel bialgebra theory.  

\begin{defi} 
	\begin{enumerate}
		\item An {\bf anti-Zinbiel coalgebra} is a vector space $A$ together
		with a linear map $\Delta: A \longrightarrow A\otimes A$ satisfying the following equations:
\begin{small} 
	\begin{align}&\label{Az1}(\tau\otimes I)(I\otimes  \Delta)\Delta=(I\otimes  \Delta)\Delta,
\\&\label{Az2}
(I\otimes \Delta)\Delta=-[(\Delta+\tau\Delta)\otimes I]\Delta=(I\otimes \tau)(\tau\otimes I)(I\otimes\tau \Delta)\Delta
=-(\tau\otimes I)(I\otimes \tau)[( \Delta+\tau \Delta)\otimes I]\Delta.\end{align} \end{small}
		\item An {\bf anti-Zinbiel bialgebra} is a triple $(A,\ast,
		\Delta)$, where $(A,\ast)$ is an anti-Zinbiel algebra and $(A,
		\Delta)$ is an anti-Zinbiel coalgebra such that for all $x,y\in A$,
	 the following compatible conditions hold:
\begin{small} 
\begin{align}&\label{Az3}\Delta(x\star y)=(I\otimes L_{\star}(x))\Delta(y)-(L_{\ast} (y)\otimes I)\Delta(x),\\&
\label{Az4}(\Delta+\tau\Delta)(x\ast y)=(L_{\ast} (x)\otimes I)(\Delta+\tau\Delta)(y)-(I\otimes R_{\ast} (y))\Delta(x),\end{align}
\end{small}
\end{enumerate}
where $x\star y=x\ast y+y\ast x$ and $\tau(x\otimes y)=y\otimes x$ for all $x,y\in A$. 
\end{defi}
 
A {\bf quadratic anti-Zinbiel algebra} $(A,\ast,\omega)$ is an anti-Zinbiel algebra together with a 
non-degenerate symmetric bilinear form $\omega$ such that Eq.~(\ref{C1}) holds.

\begin{thm} \label{Zba} Let $(A, \ast_1)$ be an anti-Zinbiel algebra
equipped with a comultiplication $\Delta:A\longrightarrow A\otimes
A$. Suppose that $\Delta^{*}:A^{*}\otimes A^{*}\longrightarrow
A^{*}$ induces an anti-Zinbiel algebra on $A^{*}$. Put
$\ast_2=\Delta^{*}$. Then the following conditions are equivalent:
\begin{enumerate}
	\item There is a quadratic anti-Zinbiel algebra $(A \oplus A^{*},\ast,\omega)$ such that 
	$(A, \ast_1)$ and $(A^{*}, \ast_2)$ are anti-Zinbiel subalgebras, where the bilinear 
	form $\omega$ on $A \oplus A^{*}$ is given by
	\begin{equation*}\omega(x + a, y + b)= \langle x,b\rangle +\langle a,y\rangle,~\forall~x,y\in A,a,b\in A^{*}.
	\end{equation*}
\item $((A,\star_1),(A^{*},\star_2),L_{\ast_1}^{*},L_{\ast_2}^{*})$ is
a matched pair of commutative associative algebras.
\item $(A,  A^{*},L_{\ast_1}^{*}+R_{\ast_1}^{*},-R_{\ast_1}^{*},L_{\ast_2}^{*}+R_{\ast_2}^{*},-R_{\ast_2}^{*})$
 is a matched pair of anti-Zinbiel algebras.
\item $(A, \ast_{1},\Delta)$ is an anti-Zinbiel bialgebra.
\end{enumerate}
\end{thm}

 Let $(A,\ast)$ be an anti-Zinbiel algebra and 
$r=\sum_{i}a_i\otimes b_i\in A\otimes A$. Denote
\begin{align*}&D(r)=r_{12}\ast r_{23}-r_{12}\ast r_{13}+r_{13}\star r_{23},
\\&D_{1}(r)=r_{12}\star r_{13}+r_{23}\ast r_{12}-r_{23}\ast r_{13},
\\&D_{2}(r)=r_{13}\ast r_{23}+r_{12}\star r_{23}+r_{13}\ast r_{12},
\\&D_{3}(r)=r_{13}\star r_{23}+r_{12}\ast r_{23}+r_{21}\ast r_{13},
\\&D_{6}(r)=r_{31}\star r_{21}+r_{32}\ast r_{21}+r_{23}\ast r_{31},
\\&D_{7}(r)=r_{13}\star r_{23}+r_{12}\ast r_{23}-r_{21}\ast r_{31},
\\&D_{8}(r)=r_{31}\star r_{21}+r_{32}\ast r_{21}-r_{23}\ast r_{13},
\\&D_{4}(r)=r_{23}\star r_{12}+r_{13}\ast r_{12}+r_{23}\ast r_{21}+r_{13}\star r_{21}+r_{13}\star r_{23},
\\&D_{5}(r)=r_{31}\star r_{21}+r_{31}\star r_{23}+r_{21}\ast r_{23}+r_{21}\star r_{32}+r_{31}\ast r_{32}.
\end{align*}

\begin{pro}\label{Zb0} Let $(A,\ast)$ be an anti-Zinbiel algebra and 
$r=\sum_{i}a_i\otimes b_i\in A\otimes A$. Define a linear map 
$\Delta_{r}:A\longrightarrow A\otimes A$ by 
\begin{equation}\label{CB1}\Delta_{r}(x)=-(I\otimes L_{\star}(x)+L_{\ast}(x)\otimes I)r,\;\forall x\in A.\end{equation}
Then we have
\begin{enumerate}
		\item Eq.~\eqref{Az1} holds if and only if the following equation holds:
\begin{equation}\label{Cd1}(L_{\ast}(x)\otimes I \otimes I-I\otimes I\otimes L_{\ast}(x))D_{3}(r)=0.\end{equation}
\item Eq.~\eqref{Az2} holds if and only if the following equations hold:
\begin{small}
\begin{align}&\label{Cd2}
(L_{\ast}(x)\otimes I \otimes I)D_{3}(r)
+I\otimes I\otimes L_{\star}(x))D_{4}(r)\\&+\sum_{i}(L_{\ast}(x\ast a_i)\otimes I\otimes I+I\otimes L_{\ast}(x\ast a_i)\otimes I)[(r+\tau(r))\otimes b_i]=0\nonumber,
\\&\label{Cd3}
(L_{\star}(x)\otimes I \otimes I)D_{5}(r)+(I\otimes I\otimes L_{\ast}(x))D_{6}(r)\\&+
\sum_{i}(I\otimes I\otimes L_{\ast}(x\ast a_i)+I\otimes L_{\ast}(x\ast a_i)\otimes I)[(r+\tau(r))\otimes b_i]=0\nonumber,
\\&\label{Cd4}
(L_{\ast}(x)\otimes I \otimes I)D_{7}(r)
-(I\otimes I\otimes L_{\ast}(x))D_{8}(r)\\&+(I\otimes L_{\ast}(x)\otimes I)(r_{32}\ast r_{21}-r_{23}\ast r_{12})
=0\nonumber.\end{align}
\end{small}
\item Eq.~\eqref{Az3} holds automatically.
\item Eq.~\eqref{Az4} holds if and only if the following equation holds:
\begin{align}&\label{Cd5}
(L_{\ast}(x)L_{\ast}(y)\otimes I -L_{\ast}(x\ast y)\otimes I-I\otimes L_{\ast}(x\ast y)+L_{\ast}(x) \otimes L_{\ast}(y))(r+\tau(r))
=0.
\end{align}
\end{enumerate}
\end{pro}

\begin{thm} \label{Zb01}
Let $(A,\ast)$ be an anti-Zinbiel algebra
and $r=\sum\limits_{i}a_{i}\otimes b_{i}\in A\otimes A$ skew-symmetric. Define a linear map
$\Delta_r:A\longrightarrow A\otimes A$ by Eq.~\eqref{CB1}.
 Then $(A,\ast,\Delta_r)$ is an anti-Zinbiel bialgebra if and only if the following equations hold:
\begin{align}&\label{Cd6}(L_{\ast}(x)\otimes I \otimes I-I\otimes I\otimes L_{\ast}(x))D(r)=0,\\
&\label{Cd7}
(L_{\ast}(x)\otimes I \otimes I+I\otimes I\otimes L_{\star}(x))D(r)=0,
\\&\label{Cd8}
(L_{\star}(x)\otimes I \otimes I)D_{1}(r)+(I\otimes I\otimes L_{\ast}(x))D_{1}(r)=0,
\\&\label{Cd9}
(L_{\ast}(x)\otimes I \otimes I)D(r)-(I\otimes I\otimes L_{\ast}(x))D_{1}(r)=0.\end{align}
\end{thm}
The equation $D(r)=r_{12}\ast r_{23}-r_{12}\ast r_{13}+r_{13}\star r_{23}=0$ is called 
the {\bf anti-Zinbiel Yang-Baxter equation} or {\bf AZ-YBE}
 in $(A,\ast)$. If $r$ is skew-symmetric, then 
 \begin{equation*}D(r)=0\Longleftrightarrow D_{1}(r)=0\Longleftrightarrow D_{2}(r)=0.\end{equation*}

\begin{thm} \label{Zb1}
Let $(A,\ast)$ be an anti-Zinbiel algebra and $r\in A\otimes A$.
If $r$ is a skew-symmetric solution of the AZ-YBE in $(A,\ast)$, then $(A,\ast,\Delta_r)$ is an anti-Zinbiel bialgebra, where
$\Delta_{r}:A\rightarrow A\otimes A$ is given by Eq.~\eqref{CB1}.
\end{thm}

\begin{defi}
	\begin{enumerate}
		\item An {\bf anti-pre-Lie coalgebra} is a vector space $A$ together
		with a linear map $\delta: A \longrightarrow A\otimes A$ such that
\begin{small}
		\begin{align}&\label{Pa1}
			(I\otimes\delta)\delta-(\tau\otimes I)(I\otimes\delta)\delta=(\tau\otimes I)(\delta\otimes I)\delta-(\delta\otimes I)\delta,\\&
\label{Pa2}[I^{\otimes 3}-\tau\otimes I+(\tau\otimes I)(I\otimes \tau)-(\tau\otimes I)(I\otimes \tau)(\tau\otimes I)
+(I\otimes \tau)(\tau\otimes I)-I\otimes \tau](\delta\otimes I)\delta=0,\end{align}
\end{small}
where $\tau: A\otimes A\longrightarrow A\otimes A$ is a map defined by $\tau(x\otimes y)=y\otimes x$ for all $x,y\in A$.
	\item An {\bf anti-pre-Lie bialgebra} is a triple $(A,\circ,
	\delta)$, where $(A,\circ)$ is an anti-pre-Lie algebra and $(A,
	\delta)$ is an anti-pre-Lie coalgebra such that for all $x,y\in A$,
	satisfying the following compatible conditions:
	\begin{small}
\begin{align}&\label{Pa3}
		\delta(x\circ y-y\circ x)=
(I\otimes \mathrm{ad}(x)-L_\circ(x) \otimes I)\delta(y)-(I\otimes \mathrm{ad}(y)-L_\circ(y) \otimes I)\delta(x),\\
		&\label{Pa4}(I\otimes I-\tau)\Big(\delta(x\circ y)-(L_\circ(x) \otimes I)\delta(y)-(I\otimes L_\circ(x))\delta(y)+(I\otimes R_\circ(y))\delta(x)\Big)=0,
	\end{align}
\end{small}
where $\mathrm{ad}(x)=L_{\circ}(x)-R_{\circ}(x)$.
	\end{enumerate}	
\end{defi}
 A {\bf quadratic anti-pre-Lie algebra} $(A,\circ,\omega)$ is an anti-pre-Lie algebra together with a 
 non-degenerate symmetric bilinear form $\omega$ such that Eq.~(\ref{C2}) holds.

\begin{thm} \label{Mb1} Let $(A, \circ_1)$ be an anti-pre-Lie algebra
	equipped with a comultiplication $\delta:A\longrightarrow A\otimes
	A$. Suppose that $\delta^{*}:A^{*}\otimes A^{*}\longrightarrow
	A^{*}$ induces an anti-pre-Lie algebra on $A^{*}$. Put
	$\circ_2=\delta^{*}$. Then the following conditions are equivalent:
	\begin{enumerate}
		\item There is a quadratic anti-pre-Lie algebra $(A \oplus A^{*},\ast,\omega)$ such that 
		$(A, \ast_1)$ and $(A^{*}, \ast_2)$ are anti-Zinbiel subalgebras, where the bilinear 
		form $\omega$ on $A \oplus A^{*}$ is given by
		\begin{equation*}\omega(x + a, y + b)= \langle x,b\rangle +\langle a,y\rangle,~\forall~x,y\in A,a,b\in A^{*}.
		\end{equation*}
		\item $((A, [ \ , \ ]_1),(A^{*},[ \ , \ ]_2), -L_{\circ_1}^{*}, -L_{\circ_2}^{*})$ is a matched pair of Lie algebras.
		\item $((A, \circ_1),( A^{*},\circ_{2}), -\mathrm{ad}_{1}^{*}, R_{\circ_1}^{*}, -\mathrm{ad}_{2}^{*}, R_{\circ_2}^{*})$ 
is a matched pair of anti-pre-Lie algebras.
		\item $(A, \circ_{1},\delta)$ is an anti-pre-Lie bialgebra.
	\end{enumerate}
\end{thm}

Let $(A,\circ)$ be an anti-pre-Lie algebra and
$r=\sum_{i}a_{i}\otimes b_{i}\in A\otimes A$. Denote
\begin{equation*}\label{Pa5}
        \begin{split}
            &S(r)=r_{12}\circ r_{13}+r_{12}\circ r_{23}-[r_{13}, r_{23}],\\&
            S_{1}(r)=[r_{13},r_{12}]-r_{23}\circ r_{12}-r_{23}\circ r_{13},
            \\&S_{2}(r)=[r_{12},r_{23}]+r_{13}\circ r_{12}-r_{13}\circ r_{23},
            \\&S_{3}(r)=r_{31}\circ r_{32}-r_{21}\circ r_{23}+[r_{31}, r_{23}]-[r_{21}, r_{32}]+[r_{12}, r_{13}],
            \\&S_{4}(r)=r_{12}\circ r_{13}-r_{32}\circ r_{31}+[r_{12}, r_{31}]-[r_{32}, r_{13}]+[r_{23}, r_{21}],
            \\&S_{5}(r)=r_{23}\circ r_{21}-r_{13}\circ r_{12}+[r_{23}, r_{12}]-[r_{13}, r_{21}]+[r_{31}, r_{32}].
        \end{split}
    \end{equation*}

\begin{pro}\label{Cg1}
Let $(A,\circ)$ be an anti-pre-Lie algebra and
$r=\sum_{i}a_{i}\otimes b_{i}\in A\otimes A$. Define a linear map
$\delta_{r}:A\rightarrow A\otimes A$ by 
\begin{equation}\label{CB2}
		\delta_{r}(x)=(L_{\circ}(x)\otimes I-I\otimes \mathrm{ad}(x))r,\;\forall x\in A.\end{equation}
Then we have
\begin{enumerate}
    \item \label{it:bb1} Eq.~(\ref{Pa1}) holds if and only if for all $x\in
    A$, the following equation holds:
  \begin{small}
\begin{equation}\label{Pa5}
        \begin{split}
            &[L_{\circ }(x)\otimes I\otimes I-(\tau\otimes I)\big(L_{\circ }(x)\otimes I\otimes I\big)
            -I\otimes I\otimes\mathrm{ad}(x)]S(r)\\
            &+\sum_{j}\big(I\otimes L_{\circ}(a_{j})\otimes\mathrm{ad}(x)
            -\mathrm{ad}(a_{j})\otimes I\otimes\mathrm{ad}(x)\big)[(r+\tau(r))\otimes b_{j}]\\
            &+(I^{\otimes 3}-\tau\otimes I)\sum_{j}\big(L_{\circ}(x\circ a_{j})\otimes I\otimes I
            -L_{\circ}(x)R_{\circ}(a_{j})\otimes I\otimes I\big)[\big(r+\tau(r)\big)\otimes b_{j}]=0.
        \end{split}
    \end{equation}\end{small}
  \begin{small}
  \item \label{it:bb2} Eq.~(\ref{Pa2}) holds if and only if for all $x\in
    A$, the following equation holds:
\begin{equation}\label{Pa6}
    \begin{split}
    &(\mathrm{ad}(x)\otimes I \otimes  I)S_{3}(r)+(I \otimes \mathrm{ad}(x)\otimes I)S_{4}(r)+(I\otimes I \otimes \mathrm{ad}(x))S_{5}(r)\\&
       +\sum_{i} (\mathrm{ad}(x\circ a_i)\otimes I \otimes  I-I \otimes\mathrm{ad}(x\circ a_i)\otimes   I)[(r+\tau(r))\otimes b_i]\\&+
        (I \otimes\mathrm{ad}(x\circ a_i)\otimes  I-I\otimes   I \otimes\mathrm{ad}(x\circ a_i))[b_i\otimes (r+\tau(r))] \\&+
        (I \otimes  I\otimes\mathrm{ad}(x\circ a_i)-\mathrm{ad}(x\circ a_i)\otimes I\otimes   I )(\tau\otimes I)[b_i\otimes (r+\tau(r))]=0.
    \end{split}\end{equation}
\end{small}
\item \label{it:bb3} Eq.~(\ref{Pa3}) holds if and only if for all $x,y\in A$, the following equation holds:
\begin{small}
    \begin{equation}\label{Pa7}
        \big(I\otimes L_{\circ}(x\circ y)-I\otimes L_{\circ}(x)L_{\circ}(y)+L_{\circ}(x)L_{\circ}(y)\otimes I-L_{\circ}(x\circ y)\otimes I+L_{\circ}(y)\otimes L_{\circ}(x)-L_{\circ}(x)\otimes L_{\circ}(y)\big)\big(r+\tau(r)\big)=0.
    \end{equation}
\end{small}
\item \label{it:bb4} Eq.~(\ref{Pa4}) holds automatically.
\end{enumerate}
\end{pro}

\begin{thm} \label{Cg2}
Let $(A,\circ)$ be an anti-pre-Lie algebra
and $r=\sum\limits_{i}a_{i}\otimes b_{i}\in A\otimes A$ skew-symmetric. Define a linear map
$\delta_r:A\longrightarrow A\otimes A$ by Eq.~(\ref{CB2}).
 Then $(A,\circ,\delta_r)$ is an anti-pre-Lie bialgebra if and only if the following equations hold:
 \begin{small}
 \begin{align}&\label{Ly1}
       [L_{\circ }(x)\otimes I\otimes I-(\tau\otimes I)\big(L_{\circ }(x)\otimes I\otimes I\big)
            -I\otimes I\otimes\mathrm{ad}(x)]S(r)=0,
\\&\label{Ly2}
 (I \otimes \mathrm{ad}(x)\otimes I)S_{2}(r)-(I\otimes I \otimes \mathrm{ad}(x))S(r)-(\mathrm{ad}(x)\otimes I \otimes  I)S_{1}(r)=0.
\end{align}\end{small}
 \end{thm}
 
If $r$ is skew-symmetric, then 
 \begin{equation*}S(r)=0\Longleftrightarrow S_{1}(r)=0\Longleftrightarrow S_{2}(r)=0.\end{equation*}
 
The equation
$S(r)=0$ is called the {\bf anti-pre-Lie Yang-Baxter equation} or {\bf APL-YBE} in $(A,\circ)$.
 
\begin{thm} \label{Pb1}
	Let $(A,\circ)$ be an anti-pre-Lie algebra and $r\in A\otimes A$. If $r$ is a skew-symmetric solution of 
the APL-YBE in $(A,\circ)$, then $(A,\circ,\delta_r)$ is an anti-pre-Lie bialgebra, where
the linear map $\delta_r:A\longrightarrow A\otimes A$ is given by Eq.~(\ref{CB2}). 
\end{thm}
 
 \subsection{Quadratic anti-pre-Poisson algebras and anti-pre-Poisson bialgebras}
 This section is devoted to presenting a bialgebra theory for anti-pre-Poisson algebras.

\begin{defi} \label{Ac1} An {\bf anti-pre-Poisson coalgebra} is a triple
$(A,\Delta,\delta)$, where $(A,\Delta)$ is an anti-Zinbiel coalgebra
and $(A,\delta)$ is an anti-pre-Lie coalgebra and $\Delta$ and
$\delta$ are compatible in the following sense:
\begin{small}
\begin{align}\label{Pc1.1}&(I\otimes \tau)(\tau\otimes I)(I\otimes\Delta)\delta
+(\tau\otimes I)(I\otimes \tau)(\tau\otimes I)(I\otimes\Delta)\delta+(I\otimes \delta)\Delta
+(\tau\otimes I)(I\otimes \delta)\Delta
\\=&(I\otimes \tau)(I\otimes \delta)\Delta+(\tau\otimes I)(I\otimes \tau)(I\otimes \delta)\Delta\nonumber,
\end{align}
\begin{equation}\label{Pc1.2}((\delta-\tau\delta)\otimes I)\Delta=(\tau\otimes I)(I\otimes\delta)\Delta-(I\otimes \Delta)\delta,\end{equation}
\begin{equation}\label{Pc1.3}((\Delta+\tau\Delta)\otimes I)\delta=-(I\otimes \delta)\Delta-(\tau\otimes I)(I\otimes\delta)\Delta.\end{equation}
\end{small}
\end{defi}

Let $A$ be a vector space with linear maps $\Delta,\delta:A\rightarrow A\otimes A$
and $\ast,\circ:A^{*}\otimes A^{*}\rightarrow A^{*}$ be the linear dual of $\Delta,\delta$ respectively. 
By a straightforward computation, $(A,\Delta,\delta)$ is 
an anti-pre-Poisson coalgebra if and only if $(A^{*},\ast,\circ)$ is an anti-pre-Poisson algebra.

\begin{defi}\label{Ac2}  Let $(A,\ast,\circ)$ be an anti-pre-Poisson algebra.
Suppose that there are two comultiplications
$\Delta,\delta:A\longrightarrow A\otimes A$ such that
$(A,\Delta,\delta)$ is an anti-pre-Poisson coalgebra. If in addition, $(A,
\ast,\Delta)$ is an anti-Zinbiel bialgebra, $(A,\circ, \delta)$ is
an anti-pre-Lie bialgebra and $\Delta,\delta$ satisfy the following
compatible conditions:
\begin{small}
\begin{align}\label{APB1}&
	\delta(x{\star}y)=(I\otimes L_{\star}(x))\delta(y)+(I\otimes L_{\star}(y))\delta(x) 
+(L_{\circ}(x)\otimes I)\Delta(y)+(L_{\circ}(y)\otimes I)\Delta(x)
	,\end{align}
\begin{align}\label{APB2}&
	\Delta([x, y])=-(L_{\circ}(x)\otimes I)\Delta(y)-(I\otimes
	L_{\star}(y))\delta(x)-(L_{\ast}(y)\otimes I)\delta(x)+(I\otimes
	\mathrm{ad}(x) )\Delta(y),\end{align}
\begin{align}\label{APB3}
	\Delta(x\circ y)+\tau\Delta(x\circ
	y)=&(L_{\circ}(x)\otimes
	I)\Delta(y)+\tau(L_{\circ}(x)\otimes I)\Delta(y)+(I\otimes
	L_{\circ}(x))\Delta(y)\\&+\tau(I\otimes
	L_{\circ}(x))\Delta(y)+(I\otimes
	R_{\ast}(y))\delta(x)+\tau(I\otimes R_{\ast}(y))\delta(x)\nonumber
	,\end{align}
\begin{align}\label{APB4}
\delta(x\ast y)-\tau\delta(x\ast y)=&(I\otimes
L_{\ast}(x))\delta(y)-\tau(L_{\ast}(x)\otimes
I)\delta(y)-(L_{\circ}(x)\otimes I)\Delta(y)\\&-\tau(I\otimes
L_{\circ}(x))\Delta(y)-(I\otimes R_{\ast}(y))\delta(x)
-\tau(I\otimes R_{\circ}(y))\Delta(x)\nonumber,\end{align}
\end{small}
for all $x,y\in A$, then $(A,\ast,\circ, \Delta,\delta)$ is called
an {\bf anti-pre-Poisson bialgebra}.\end{defi}

\begin{rmk} $(A,\Delta,\delta)$ is 
an anti-pre-Poisson coalgebra 
 if and only if $(A^{*},\ast_{A^{*}},\circ_{A^{*}})$
 is an anti-pre-Poisson algebra, where $\ast_{A^{*}},\circ_{A^{*}}$ are the linear dual of $\Delta,\delta$
 respectively, that is,
 \begin{align}&\label{Dm1}\langle \Delta(x),\zeta\otimes \eta\rangle=\langle x,\zeta\ast_{A^{*}} \eta\rangle
\\&\label{Dm2}\langle \delta(x),\zeta\otimes \eta\rangle=\langle x,\zeta\circ_{A^{*}} \eta\rangle,~\forall~x\in A,\zeta,\eta\in A^{*}.
\end{align}
 Thus, an anti-pre-Poisson bialgebra $(A,\ast,\circ,\Delta,\delta)$ is sometimes
denoted by $(A,\ast,\circ,A^{*},\ast_{A^{*}},\circ_{A^{*}})$, where the anti-pre-Poisson
 algebra structure $(A^{*},\ast_{A^{*}},\circ_{A^{*}})$
 on the dual space $A^{*}$
corresponds to the anti-pre-Poisson coalgebra $(A,\Delta,\delta)$
through Eqs.~(\ref{Dm1})-(\ref{Dm2}).
\end{rmk}

\begin{defi} Let $(A, \ast,\circ)$ be an anti-pre-Poisson
algebra. If there exists a non-degenerate symmetric bilinear form 
$\omega$ on $A$ such that $(A,\ast,\omega)$ is a quadratic anti-Zinbiel algebra and $(A,\circ,\omega)$ is a quadratic anti-pre-Lie algebra, 
then we say that $(A,\ast,\circ,\omega)$ is a {\bf quadratic anti-pre-Poisson algebra}.
\end{defi}

\begin{pro}\label{Qc} Let $(A, \ast, \circ,\omega)$ be a quadratic anti-pre-Poisson algebra, and $(A, \star, [ \ , \ ])$ be
its sub-adjacent Poisson algebra. Then $\omega$ is both a commutative Connes cocycle on the 
associative algebra $(A, \star)$ and a commutative 2-cocycle on the Lie algebra $(A,  [ \ , \ ])$. 
Conversely,
assume that $(A, \star, [ \ , \ ])$ is a Poisson algebra and $\omega$ is
 both a commutative Connes cocycle on the associative algebra $(A, \star)$
  and a commutative 2-cocycle on the Lie algebra $(A,  [ \ , \ ])$. Then
there exists a compatible quadratic anti-pre-Poisson algebra $(A, \ast, \circ,\omega)$
 given by Eqs.~(\ref{C1}) and (\ref{C2}).\end{pro}
 
 \begin{proof}
 Combining Proposition \ref{Qa1} and Proposition \ref{Qa2}, we get the statement.
 
  \end{proof}
\begin{thm} \label{Mb3} Let $(A, \ast_1,\circ_1)$ be an anti-pre-Poisson algebra
equipped with two comultiplications $\Delta,\delta:A\longrightarrow
A\otimes A$. Suppose that $\Delta^{*},\delta^{*}:A^{*}\otimes
A^{*}\longrightarrow A^{*}$ induce an anti-pre-Poisson algebra structure
on $A^{*}$. Put $\ast_2=\Delta^{*},\circ_2=\delta^{*}$. Then the
following conditions are equivalent:
\begin{enumerate}
	\item\label{equiva1} There is a quadratic anti-pre-Poisson algebra $(A \oplus A^{*},\ast,\circ,\omega )$ such that 
	$(A, \ast_1,\circ_1)$ and $(A^{*}, \ast_2$,
	$\circ_2)$ are anti-pre-Poisson subalgebras, where the bilinear 
	form $\omega$ on $A \oplus A^{*}$ is given by
	\begin{equation}\label{eq:om}
		\omega(x + a, y + b)= \langle x,b\rangle +\langle a,y\rangle,~\forall~x,y\in A,a,b\in A^{*}.
	\end{equation}
\item\label{equiva2}
$(A, A^{*}, L_{\ast_1}^{*},-L_{\circ_1}^{*},L_{\ast_2}^{*},-L_{\circ_2}^{*}) $ is a matched pair of Poisson algebras.
\item\label{equiva3}
$(A, A^{*}, L_{\ast_1}^{*}+R_{\ast_1}^{*}, -R_{\ast_1}^{*}, R_{\circ_1}^{*}-L_{\circ_1}^{*},R_{\circ_1}^{*},L_{\ast_2}^{*}+R_{\ast_2}^{*}, -R_{\ast_2}^{*},
R_{\circ_2}^{*}-L_{\circ_2}^{*},R_{\circ_2}^{*})$ is a matched pair of anti-pre-Poisson algebras.
	\item\label{equiva4} $(A, \ast_{1},\circ_{1},\Delta,\delta)$ is an anti-pre-Poisson bialgebra.
\end{enumerate}
\end{thm}

\begin{proof} 
By Theorem \ref{Mb1}, Theorem \ref{Mb3} and Proposition \ref{Mb2}, it follows directly that
\eqref{equiva1} $\Longleftrightarrow$ \eqref{equiva2} $\Longleftrightarrow$ \eqref{equiva3}.
It remains to show that \eqref{equiva2} $\Longleftrightarrow$ \eqref{equiva4}.
In view of Theorem \ref{Mb1} and Theorem \ref{Mb3}, it suffices to prove the following equivalences under the conditions
\begin{align*}
	\mu_{1}=L_{\ast_{1}}^{*}, \ \ \ \rho_{1}=-L_{\circ_{1}}^{*}, \ \ \
	\mu_{2}=L_{\ast_{2}}^{*}, \ \ \ \rho_{2}=-L_{\circ_{2}}^{*}.
\end{align*}
 In fact, for all $x,y\in A,a,b\in A^{*}$, we obtain
\begin{align*}
\langle [a,L_{\ast_{1}}^{*}(x)b]_{2},y\rangle&=\langle a\circ_2 (L_{\ast_{1}}^{*}(x)b)-(L_{\ast_{1}}^{*}(x)b)\circ_2 b,y
\rangle
\\&=-\langle (I\otimes L_{\ast_{1}}(x))(\delta(y)-\tau\delta(y)),a\otimes b\rangle,\end{align*}
\begin{align*}
\langle L_{\circ_{1}}^{*}(L_{\ast_{2}}^{*}(b)x)a,y
\rangle &=-\langle a,(L_{\ast_{2}}^{*}(b)x)\circ_{1}y
\rangle=\langle R_{\circ_{1}}^{\ast}(y)a,L_{\ast_{2}}^{*}(b)x
\rangle\\&=-\langle b\ast_{2}R_{\circ_{1}}^{\ast}(y)a,x
\rangle=\langle a\otimes b,(R_{\circ_{1}}(y)\otimes I)\tau\Delta(x)\rangle,
\end{align*}
\begin{align*}
\langle -L_{\ast_{1}}^{*}(L_{\circ_{2}}^{*}(a)x)b,y\rangle&=\langle b,(L_{\circ_{2}}^{*}(a)x)\ast_{1}y\rangle
=\langle -R_{\ast_{1}}^{*}(y)b,L_{\circ_{2}}^{*}(a)x\rangle
\\&=\langle a\circ_{2}(R_{\ast_{1}}^{*}(y)b),x\rangle
=-\langle (I\otimes R_{\ast_{1}}(y))\delta(x),a\otimes b,\rangle,\end{align*}
\begin{align*}
\langle (L_{\circ_{1}}^{*}(x)a)\star_{2}b+L_{\ast_{1}}^{*}(x)([a,b]_{2}),y
\rangle
&=\langle (\Delta^{*}+\Delta^{*}\tau)((L_{\circ_{1}}^{*}(x)a)\otimes b),y
\rangle-\langle [a,b]_{2},x\ast_{1}y\rangle
\\&=-\langle a\otimes b,(L_{\circ_{1}}(x)\otimes I)(\Delta +\tau \Delta)(y)+(\tau\delta- \delta)(x\ast_{1}y)
\rangle.
\end{align*}
Thus, $(\ref{P4})\Longleftrightarrow (\ref{APB4})$. Similarly, $(\ref{P1})\Longleftrightarrow (\ref{APB1}),~
(\ref{P2})\Longleftrightarrow (\ref{APB2}),~(\ref{P3})\Longleftrightarrow (\ref{APB3})$.
The proof is completed.
\end{proof}

Let $(A, \ast,\circ,\Delta, \delta)$ be
 an anti-pre-Poisson bialgebra. Then $(D=A\oplus A^{*},\ast_{D},\circ_{D})$
 is an anti-pre-Poisson algebra, where
 \begin{align}\label{Db1}(x+a)\ast_{D}(y+b)=&x\ast_A y+(L_{\ast_{A^*}}^{*}+R_{\ast_{A^*}}^{*})
(a)y-R_{\ast_{A^*}}^*
(b)x\\&+a\ast_{A^*}b+(L_{\ast_A}^{*}+R_{\ast_A}^{*}(x)b-R_{\ast_A}^{*}(y)a\nonumber
,\end{align}
\begin{align}\label{Db2}
(x+a)\circ_{D}(y+b)=&x\circ_A y+(R_{\circ_{A^*}}^{*}-L_{\circ_{A^*}}^{*})
(a)y+R_{\circ_{A^*}}^{*}(b)x\\&+a\circ_{A^*}b+(R_{\circ_A}^{*}-L_{\circ_A}^{*})
(x)b+R_{\circ_A}^{*}(y)a\nonumber,
\end{align}
for all $x,y\in A,a,b\in A^{*}$.
$(D=A\oplus A^{*},\ast_{D},\circ_{D})$ is called the double anti-pre-Poisson algebra. 

%%%%%%%%%%%%%%%%%%%%%%%%%%%%%%%%%%%%%%%%%%%%%%%%%%%%%%%%%%%%%%%%%%%%%%%%%%%%%%%%%%%%%%%%%%%%%%%%%%%%%%%%%%

\subsection{Coboundary anti-pre-Poisson bialgebras.}

In this section, we consider the coboundary anti-pre-Poisson
bialgebras. 

\begin{defi} An anti-pre-Poisson bialgebra $(A,\ast,\circ,\Delta,\delta)$ is called coboundary if there is some
$r\in A\otimes A$ such that Eqs.~\eqref{CB1} and \eqref{CB2} hold.
\end{defi}

\begin{pro} \label{Ac3}
Let $(A,\ast,\circ,\Delta_r,\delta_r)$ be an anti-pre-Poisson bialgebra
and $r=\sum\limits_{i}a_{i}\otimes b_{i}\in A\otimes A$, where
$\Delta_r,\delta_r:A\longrightarrow A\otimes A$ are given by Eqs.~(\ref{CB1}) and
(\ref{CB2}) respectively. Then
\begin{enumerate}
\item \label{CNPB1} Eq.~(\ref{Pc1.1}) holds if and only if
\begin{small}
\begin{align}\label{CA1}&
\sum_{i}(L_{\ast}([x,b_{i}])\otimes I \otimes I)[(r+\tau(r))\otimes a_{i}]
-(L_{\ast}(x\star b_i)\otimes I \otimes I)(\tau\otimes I)[a_i\otimes (r+\tau(r))]
\\&+( I\otimes I \otimes L_{\circ}(x\star b_i))[a_i\otimes (r+\tau(r))]
+( I\otimes I \otimes L_{\circ}(x\star b_i))(\tau\otimes I)[a_i\otimes (r+\tau(r))]
\nonumber\\&+( I \otimes L_{\ast}([x,b_i])\otimes I)[(r+\tau(r))\otimes a_{i}]
-( I \otimes L_{\circ}(x\star b_i)\otimes I)[a_i\otimes (r+\tau(r))]
\nonumber\\&+(L_{\ast}(x)\otimes I \otimes I)([r_{32},r_{12}]+r_{13}\circ r_{32}
+[r_{13},r_{23}]-r_{12}\circ r_{23}+[r_{31},r_{21}])
\nonumber\\&+( I \otimes L_{\ast}(x)\otimes I)([r_{31},r_{21}]+r_{23}\circ r_{31}+[r_{23},r_{13}]-r_{21}\circ r_{13}+[r_{32},r_{12}])
\nonumber\\&+(I\otimes I\otimes L_{\circ}(x))(r_{13}\star r_{23}-r_{32}\star r_{12}-r_{31}\ast r_{12}-r_{31}\star r_{21}-r_{32}\ast r_{21})
=0\nonumber.\end{align}\end{small}
\item \label{CNPB2}
	 Eq.~(\ref{Pc1.2}) holds if and only if
\begin{small}
\begin{align}\label{CA2}&
(I\otimes I\otimes L_{\star} (x))([r_{23}, r_{12}]-r_{13}\circ r_{12}+r_{23}\circ r_{21}-[r_{13}, r_{21}]+[r_{13}, r_{23}])	
\\&-(L_{\circ} (x\ast a_i)\otimes I\otimes I)((r+\tau(r))\otimes b_i)
+(I\otimes L_{\circ} (x\ast a_i)\otimes I)((r+\tau(r))\otimes b_i)
\nonumber\\&-(L_{\circ} (x)\otimes I\otimes I)(r_{13}\star r_{23}+r_{12}\ast r_{23}+r_{21}\ast r_{13})
\nonumber\\&+(I\otimes L_{\ast}(x)\otimes I)(r_{21}\circ r_{13} -r_{12}\circ r_{23}+[r_{13}, r_{23}])=0\nonumber
\end{align}\end{small}
\item \label{CNPB3} Eq.~(\ref{Pc1.3}) holds if and only if
\begin{small}
\begin{align}\label{CA3}&
(I\otimes I\otimes \mathrm{ad}(x))(r_{13}\ast r_{12}+r_{23}\star r_{12}+r_{13}\star r_{21}+r_{23}\ast r_{21}+r_{13}\star r_{23})\\&
+(L_{\ast}(x)\otimes I\otimes I-I\otimes L_{\ast}(x)\otimes I)([r_{13}, r_{23}]-r_{12}\circ r_{23}+r_{21}\circ r_{13})
\nonumber\\&-(L_{\ast} (x\circ a_i)\otimes I\otimes I+ I\otimes L_{\ast} (x\circ a_i)\otimes  I)[(r+\tau(r))\otimes b_i]
=0.\nonumber\end{align}\end{small}
\item \label{CNPB4} Eqs.~(\ref{APB1}) and (\ref{APB2}) hold naturally.
\item \label{CNPB5} Eq.~(\ref{APB3}) holds if and only if
\begin{small}
	\begin{align}\label{CA4}
		&(I\otimes L_{\circ}(x)L_{\ast}(y)-I\otimes L_{\ast}(x\circ y)-L_{\ast}(x\circ y)\otimes I
+L_{\circ}(x)L_{\ast}(y)\otimes I 
\\&+L_{\circ}(x)\otimes L_{\ast}(y)+L_{\ast}(y)\otimes L_{\circ}(x))(r+\tau(r))=0.\nonumber\end{align}
			\item \label{CNPB6} Eq.~(\ref{APB4}) holds if and only if
		\begin{align}\label{CA5}&
			(I\otimes L_{\ast}(x)L_{\circ}(y)-I\otimes L_{\circ}(x\ast y)-L_{\circ}(x)\otimes L_{\ast}(y)-L_{\circ}(y)\otimes L_{\ast}(x)
	\\&-L_{\circ}(x)L_{\ast}(y)\otimes I+L_{\circ}(x\ast y)\otimes I)(r+\tau(r))=0\nonumber.\end{align}\end{small}
\end{enumerate}
\end{pro}

\begin{proof} Using Eq.~(\ref{ppa eq1.3}), we get
\begin{align*}&
b_{j}\ast[x,b_{i}]
+b_i\circ(x\star b_j)=x\ast[b_{i},b_{j}],
\\&b_{j}\circ(x\star b_i)+b_{i}\circ(x\star b_j)=-x\circ(b_{i}\star b_j).
\end{align*}

Combining Eqs.~\eqref{CB1} and \eqref{CB2}, we have for all $x\in A$, 
\begin{small}
\begin{align*}&(I\otimes \tau)(\tau\otimes I)[I\otimes(\Delta_r+\tau \Delta_r)]\delta_r(x)
-(I\otimes \tau)(I\otimes \delta_r)\Delta_r(x)-(\tau\otimes I)(I\otimes \tau)(I\otimes \delta_r)\Delta_{r}(x)
\\&+(I\otimes \delta_r)\Delta_r(x)
+(\tau\otimes I)(I\otimes \delta_r)\Delta_r(x)
\\=&\sum_{i,j}a_{j}\otimes [x,b_{i}]\star b_{j}\otimes a_{i}+[x,b_{i}]\ast a_{j}\otimes b_{j}\otimes a_{i}-
a_j\otimes b_i\star b_j\otimes x\circ a_i-b_i\ast a_j\otimes b_{j}\otimes x\circ a_{i}\\&+
[x,b_{i}]\star b_{j}\otimes a_{j}\otimes a_{i}
+ b_{j}\otimes[x,b_{i}]\ast a_{j}\otimes a_{i}
-b_i\star b_j\otimes a_j\otimes  x\circ a_i-b_{j}\otimes b_i\ast a_j\otimes  x\circ a_{i}\\&-
a_i\otimes [x\star b_i,b_j]\otimes a_{j}+a_i\otimes b_j \otimes(x\star b_i)\circ a_j-
x\ast a_i \otimes[b_i,b_j]\otimes a_{j}+x\ast a_i \otimes b_{j}\otimes b_i\circ a_{j}\\&
-[x\star b_i,b_j]\otimes a_{i}\otimes a_{j}+b_{j}\otimes a_i \otimes(x\star b_i)\circ a_j
-[b_i,b_j]\otimes x\ast a_i \otimes a_j+b_{j}\otimes x\ast a_i \otimes b_i\circ a_{j}
\\&+a_i \otimes a_j \otimes [x\star b_i, b_{j}]-a_i\otimes(x\star b_i)\circ a_{j}\otimes b_j
+x\ast a_i \otimes a_{j}\otimes [b_i,b_j]-x\ast a_i \otimes  b_i\circ a_{j}\otimes b_j
\\&+ a_j \otimes a_i \otimes [x\star b_i, b_{j}]-(x\star b_i)\circ a_{j}\otimes a_i\otimes b_j
+a_{j}\otimes x\ast a_i \otimes  [b_i,b_j]-b_i\circ a_{j}\otimes x\ast a_i \otimes   b_j
\\=&A(1)+A(2)+A(3),
\end{align*}\end{small}
where
\begin{small}
\begin{align*}&A(1)=
\sum_{i,j}[x,b_{i}]\ast a_{j}\otimes b_{j}\otimes a_{i}+
[x,b_{i}]\star b_{j}\otimes a_{j}\otimes a_{i}
-x\ast a_i \otimes[b_i,b_j]\otimes a_{j}+x\ast a_i \otimes b_{j}\otimes b_i\circ a_{j}
\\&-[x\star b_i,b_j]\otimes a_{i}\otimes a_{j}
+x\ast a_i \otimes a_{j}\otimes [b_i,b_j]-x\ast a_i \otimes  b_i\circ a_{j}\otimes b_j
-(x\star b_i)\circ a_{j}\otimes a_i\otimes b_j
\\=&\sum_{i}(L_{\ast}([x,b_{i}])\otimes I \otimes I)[(r+\tau(r))\otimes a_{i}]+
 b_{j}\ast[x,b_{i}]\otimes a_{j}\otimes a_{i}+(L_{\ast}(x)\otimes I \otimes I)([r_{32},r_{12}]\\&+r_{13}\circ r_{32})
-(L_{\ast}(x\star b_j)\otimes I \otimes I)(\tau\otimes I)(a_j\otimes \tau(r))
+b_i\circ(x\star b_j)\otimes a_{j}\otimes a_{i}\\&
+(L_{\ast}(x)\otimes I \otimes I)([r_{13},r_{23}]-r_{12}\circ r_{23})-
(L_{\ast}(x\star b_j)\otimes I \otimes I)(\tau\otimes I)(a_j\otimes r)
\\=&\sum_{i}(L_{\ast}([x,b_{i}])\otimes I \otimes I)[(r+\tau(r))\otimes a_{i}]
-(L_{\ast}(x\star b_j)\otimes I \otimes I)(\tau\otimes I)[a_j\otimes (r+\tau(r))]
\\&+(L_{\ast}(x)\otimes I \otimes I)([r_{32},r_{12}]+r_{13}\circ r_{32}
+[r_{13},r_{23}]-r_{12}\circ r_{23}+[r_{31},r_{21}]),\end{align*}
\begin{align*}&A(2)=
\sum_{i,j}a_{j}\otimes [x,b_{i}]\star b_{j}\otimes a_{i}
+ b_{j}\otimes[x,b_{i}]\ast a_{j}\otimes a_{i}
-a_i\otimes [x\star b_i,b_j]\otimes a_{j}
-[b_i,b_j]\otimes x\ast a_i \otimes a_j\\&+b_{j}\otimes x\ast a_i \otimes b_i\circ a_{j}
-a_i\otimes(x\star b_i)\circ a_{j}\otimes b_j
+a_{j}\otimes x\ast a_i \otimes  [b_i,b_j]-b_i\circ a_{j}\otimes x\ast a_i \otimes   b_j
\\=&
( I \otimes L_{\ast}([x,b_i])\otimes I)(r\otimes a_i)+
a_{j}\otimes b_{j}\ast[x,b_{i}] \otimes a_{i}+( I \otimes L_{\ast}([x,b_i])\otimes I)(\tau(r)\otimes a_i)
\\&-( I \otimes L_{\circ}(x\star b_i)\otimes I)(a_i\otimes \tau(r))+a_i\otimes b_j\circ(x\star b_i)\otimes a_{j}
+( I \otimes L_{\ast}(x)\otimes I)([r_{31},r_{21}]+r_{23}\circ r_{31})
\\&-( I \otimes L_{\circ}(x\star b_i)\otimes I)(a_i\otimes r)
+( I \otimes L_{\ast}(x)\otimes I)([r_{23},r_{13}]-r_{21}\circ r_{13})
\\=&( I \otimes L_{\ast}([x,b_i])\otimes I)[(r+\tau(r))\otimes a_{i}]
-( I \otimes L_{\circ}(x\star b_i)\otimes I)[a_i\otimes (r+\tau(r))]
\\&+( I \otimes L_{\ast}(x)\otimes I)([r_{31},r_{21}]+r_{23}\circ r_{31}+[r_{23},r_{13}]-r_{21}\circ r_{13}+[r_{32},r_{12}]),\end{align*}
\begin{align*}&A(3)=
\sum_{i,j}-
a_j\otimes b_i\star b_j\otimes x\circ a_i-b_i\ast a_j\otimes b_{j}\otimes x\circ a_{i}
-b_i\star b_j\otimes a_j\otimes  x\circ a_i-b_{j}\otimes b_i\ast a_j\otimes  x\circ a_{i}\\&
+a_i\otimes b_j \otimes(x\star b_i)\circ a_j+b_{j}\otimes a_i \otimes(x\star b_i)\circ a_j
+a_i \otimes a_j \otimes [x\star b_i, b_{j}]
+ a_j \otimes a_i \otimes [x\star b_i, b_{j}]
\\=&(I\otimes I\otimes L_{\circ}(x))(-r_{32}\star r_{12}-r_{31}\ast r_{12}-r_{31}\star r_{21}-r_{32}\ast r_{21})
+( I\otimes I \otimes L_{\circ}(x\star b_i))(a_i\otimes \tau(r))\\&
+( I\otimes I \otimes L_{\circ}(x\star b_i))(\tau\otimes I)(a_i\otimes \tau(r))+
+( I\otimes I \otimes L_{\circ}(x\star b_i))(a_i\otimes r)\\&
-a_i \otimes a_j \otimes b_{j}\circ(x\star b_i)
+( I\otimes I \otimes L_{\circ}(x\star b_i))(\tau\otimes I)(a_i\otimes r)
-a_i \otimes a_j \otimes b_{i}\circ(x\star b_j)
\\=&(I\otimes I\otimes L_{\circ}(x))(r_{13}\star r_{23}-r_{32}\star r_{12}-r_{31}\ast r_{12}-r_{31}\star r_{21}-r_{32}\ast r_{21})
\\&+( I\otimes I \otimes L_{\circ}(x\star b_i))[a_i\otimes (r+\tau(r))]
+( I\otimes I \otimes L_{\circ}(x\star b_i))(\tau\otimes I)[a_i\otimes (r+\tau(r))]
.\end{align*}\end{small}
Therefore, Eq.~(\ref{Pc1.1}) holds if and only if Eq.~(\ref{CA1}) holds. The remaining items can be checked in the same way. \end{proof}

\begin{thm} \label{Bc1} Let $(A,\ast,\circ)$ be an anti-pre-Poisson algebra equipped with
linear maps $\Delta_r,\delta_r:A\longrightarrow A\otimes A$ given by Eqs.~\eqref{CB1} and \eqref{CB2} respectively. 
Then $(A,\ast,\circ,\Delta_r,\delta_r)$
 is an anti-pre-Poisson bialgebra if and only if 
 Eqs.~\eqref{Cd1}-\eqref{Cd4}, \eqref{Pa5}-\eqref{Pa7} and \eqref{CA1}-\eqref{CA5} hold.\end{thm}
\begin{proof} This follows from Proposition \ref{Zb0}, Proposition \ref{Cg1} and Proposition \ref{Ac3}.\end{proof}

The following conclusion is apparent.

\begin{pro} \label{Ac4}
Let $(A,\ast,\circ)$ be an anti-pre-Poisson algebra and $r=\sum\limits_{i}a_{i}\otimes b_{i}\in A\otimes A$ skew-symmetric.
Define linear maps $\Delta_r,\delta_r:A\longrightarrow A\otimes A$ by Eqs.~(\ref{CB1}) and (\ref{CB2}) respectively
Then $(A,\ast,\circ,\Delta_r,\delta_r)$
 is an anti-pre-Poisson bialgebra if and only if
 Eqs.~(\ref{Cd6})-(\ref{Cd9}), (\ref{Ly1})-(\ref{Ly2}) and
  the following equations hold:
 \begin{small}
\begin{align}&\label{Py1}
(L_{\circ} (x)\otimes I\otimes I)D(r)-(I\otimes I\otimes L_{\star} (x)+I\otimes L_{\ast}(x)\otimes I)S(r)
=0,
\\&\label{Py2}
(I\otimes I\otimes \mathrm{ad}(x))D(r)
-(L_{\ast}(x)\otimes I\otimes I-I\otimes L_{\ast}(x)\otimes I)S(r)
=0,\\&\label{Py3}
( I \otimes L_{\ast}(x)\otimes I)S_{2}(r)+(L_{\ast}(x)\otimes I \otimes I)S_{1}(r)
+(I\otimes I\otimes L_{\circ}(x))D(r)
=0.\end{align}\end{small}
\end{pro}

\begin{rmk} \label{Ac6} Assume that $\sigma_{123},\sigma_{23}:A\otimes A\otimes A\longrightarrow A\otimes A\otimes A$ 
are linear maps defined respectively by
\begin{equation*}\sigma_{123}(x\otimes y\otimes z)=z\otimes x\otimes y, \ \ \
\sigma_{23}(x\otimes y\otimes z)=x\otimes z\otimes y, \ \ \forall~x,y,z\in A. \end{equation*}
If $r$ is skew-symmetric, then
\begin{align*}&\sigma_{123}S(r)=S_{1}(r),\ \ \  \sigma_{23}S(r)=S_{2}(r).\end{align*}
\end{rmk}

\begin{thm} \label{Ac5} Let $(A,\ast,\circ)$ be an anti-pre-Poisson algebra
and $r\in A\otimes A$ be skew-symmetric.
Then $(A,\ast,\circ,\Delta_r,\delta_r)$ is an anti-pre-Poisson bialgebra if
$S(r)=D(r)=0$, where $\Delta_r,\delta_r$ are defined by Eqs.~\eqref{CB1} and \eqref{CB2} respectively.
\end{thm}
\begin{proof} 
Combining Theorem \ref{Zb1}, Theorem \ref{Pb1}, Theorem \ref{Bc1}, Proposition \ref{Ac4} and Remark \ref{Ac6}, we can get the statement.
\end{proof}

\begin{defi} 
Let $(A,\ast,\circ)$ be an anti-pre-Poisson algebra and $r\in
A\otimes A$. We say that  $r$ satisfies {\bf the anti-pre-Poisson Yang-Baxter equation} or {\bf APP-YBE} in short
 if $r$ satisfies both the AZ-YBE:
\begin{equation}D(r)=r_{12}\ast r_{23}-r_{12}\ast r_{13}+r_{13}\star r_{23}=0\end{equation}
 and the APL-YBE:
\begin{equation} S(r)=r_{12}\circ r_{23}+r_{12}\circ r_{13}-[r_{13}, r_{23}]=0.\end{equation}
 \end{defi}
 
 Thus, the following conclusion is reached.
 
\begin{cor}\label{Bc2} 
	Let $(A,\ast,\circ)$ be an anti-pre-Poisson algebra and $r\in A\otimes A$ be
	a skew-symmetric solution of the APP-YBE in $(A,\ast,\circ)$.
	Then $(A,\ast,\circ,\Delta_r,\delta_r)$ is an anti-pre-Poisson bialgebra, where $\Delta_r,\delta_r:A\rightarrow A\otimes A$ are 
linear maps given by Eqs.~\eqref{CB1} and \eqref{CB2} respectively.
\end{cor}

\begin{ex} Let $(A,\ast,\circ)$ be the 3-dimensional anti-pre-Poisson algebra given in Example \ref{Ae}
with a basis $\{e_1, e_2, e_3 \}$. Then 
\begin{equation*}r=(e_1+e_2)\otimes e_3-e_3\otimes (e_1+e_2)\end{equation*} is a skew-symmetric 
solution of the APP-YBE in $(A,\ast,\circ)$. Thus, $(A,\ast,\circ,\Delta_r,\delta_r)$ is an anti-pre-Poisson bialgebra, 
where $\Delta_r,\delta_r:A\rightarrow A\otimes A$ are 
linear maps given by Eqs.~\eqref{CB1} and \eqref{CB2} respectively. Explicitly,
\begin{align*}& \Delta(e_{1})=e_3\otimes e_3,\ \ \Delta(e_{2})=\Delta(e_{3})=0,\\
&\delta(e_{1})=2e_3\otimes e_3,\ \ \delta(e_{2})=-e_3\otimes e_3,\ \
\delta(e_{3})=0.\end{align*}
\end{ex}

For a vector space $A$, the isomorphism $A\otimes A^{*}\simeq Hom (A^{*},A)$ identifies an element $r\in A\otimes A$ with a map
$T_{r}:A^{*}\longrightarrow A$. Explicitly, 
\begin{equation} \label{YE7} T_{r}:A^{*}\longrightarrow A,\ \ \  \langle T_{r}(\zeta),\eta\rangle=\langle r,\zeta\otimes\eta\rangle,
\ \ \ \forall~\zeta,\eta\in A^{*}.\end{equation}
It is clearly that $T_{r}^{*}=T_{\tau(r)},~T_{r+\tau(r)}^{*}=T_{r+\tau(r)}$.

\begin{pro} \label{Ys1} Let $(A,\circ)$ be an anti-pre-Lie
 algebra and 
$r=\sum_{i}a_i\otimes b_i\in A\otimes A$. Then the following conclusions hold:
\begin{enumerate}
		\item $r$ is a solution of the APL-YBE $S(r)=r_{12}\circ r_{23}+r_{12}\circ r_{13}-[r_{13}, r_{23}]=0$
if and only if
\begin{equation*}T_{\tau(r)}(\zeta)\circ T_{\tau(r)}(\eta)=T_{\tau(r)}( \mathrm{ad}^*(T_{r}(\zeta))\eta+R_{\circ}^*(T_{\tau(r)}(\eta))\zeta).\end{equation*}
\item $r$ is a solution of the APL-YBE $S(r)=r_{12}\circ r_{23}+r_{12}\circ r_{13}-[r_{13}, r_{23}]=0$
if and only if
\begin{equation*}[T_{r}(\zeta), T_{r}(\eta)]=T_{r}( L_{\circ}^*(T_{\tau(r)}(\zeta))\eta+L_{\circ}^*(T_{r}(\eta))\zeta).\end{equation*}
\item $r$ is a solution of the equation
$S_{2}(r)=r_{13}\circ r_{12}+[r_{12}, r_{23}]-r_{13}\circ r_{23}=0$
if and only if
\begin{equation*}T_{r}(\zeta)\circ T_{r}(\eta)=-T_{r}(\mathrm{ad}^*(T_{r}(\zeta))\eta+R_{\circ}^*(T_{\tau(r)}(\eta))\zeta).\end{equation*}
\item $r$ is a solution of the equation
$S_{1}(r)=[r_{13},r_{12}]-r_{23}\circ r_{12}-r_{23}\circ r_{13}=0$
if and only if
\begin{equation*}[T_{\tau(r)}(\zeta), T_{\tau(r)}(\eta)]=-T_{\tau(r)}( L_{\circ}^*(T_{\tau(r)}(\zeta))\eta+L_{\circ}^*(T_{r}(\eta))\zeta)
.\end{equation*}
\end{enumerate}
\end{pro}

\begin{proof} According to Eq.~\eqref{YE7}, we have for all $\zeta,\eta,\theta\in A^{*}$,
\begin{small}
\begin{align*}\langle\theta\otimes \zeta\otimes \eta,r_{12}\circ r_{13}\rangle
&=\sum_{i,j}\langle \theta\otimes \zeta\otimes \eta,a_i\circ a_j\otimes b_i\otimes b_j\rangle
\\&=\sum_{i,j}\langle \theta,a_i\circ a_j\rangle\langle \zeta,b_i\rangle \langle \eta,b_j\rangle 
=\langle T_{\tau(r)}(\zeta)\circ T_{\tau(r)}(\eta),\theta \rangle
,\end{align*}
\begin{align*}\langle \theta\otimes \zeta\otimes \eta,[r_{23},r_{13}]\rangle
&=\sum_{i,j}\langle \theta\otimes \zeta\otimes \eta,a_i \otimes a_j\otimes [b_j, b_i]\rangle
\\&=\sum_{i,j}\langle \zeta,a_j\rangle \langle \theta,a_i\rangle\langle \eta,[b_j, b_i]\rangle 
=\langle \eta,[T_{r}(\zeta), T_{r}(\theta)]\rangle
\\&=-\langle \mathrm{ad}^{*}(T_{r}(\zeta))\eta, T_{r}(\theta)\rangle
=-\langle T_{\tau(r)}(\mathrm{ad}^{*}(T_{r}(\zeta))\eta), \theta\rangle,\end{align*}
 \begin{align*}\langle \theta\otimes \zeta\otimes \eta,r_{12}\circ r_{23}\rangle
 &=\sum_{i,j}\langle \theta\otimes \zeta\otimes \eta,a_i \otimes (b_i\circ a_j)\otimes b_j\rangle
 \\& =\sum_{i,j}\langle \theta,a_i\rangle\langle \eta,b_j\rangle \langle \zeta,b_i\circ a_j\rangle 
 =\langle \zeta,T_{r}(\theta)\circ T_{\tau(r)}(\eta)\rangle
  \\&=-\langle R_{\circ}^{*}(T_{\tau(r)}(\eta))\zeta,T_{r}(\theta)\rangle
 =-\langle T_{\tau(r)}(R_{\circ}^{*}(T_{\tau(r)}(\eta))\zeta),\theta\rangle.\end{align*}\end{small}
 Thus, Item (a) holds. The other items can be verified similarly.
\end{proof}

\begin{pro} \cite{19} \label{Ys2} Let $(A,\ast)$ be an anti-Zinbiel algebra and
$r\in A\otimes A$. Then  \begin{enumerate}
\item
 $r$ is a solution of the equation $ D(r)=r_{13}\star r_{23}-r_{12}\ast r_{13}+r_{12}\ast r_{23}=0$
 in $(A,\ast,\circ)$
if and only if the following equation holds:
 \begin{equation*}T_{\tau(r)}(\eta)\ast T_{\tau(r)}(\zeta)=-T_{\tau(r)}(R_{\ast}^{*}(T_{\tau(r)}(\zeta))\eta
+L_{\star}^{*}(T_{r}(\eta))\zeta),~~\forall~\eta,\zeta\in A^{*}.
\end{equation*}
\item
 $r$ is a solution of the equation $ D(r)=r_{13}\star r_{23}-r_{12}\ast r_{13}+r_{12}\ast r_{23}=0$
 in $(A,\succ,\prec)$
if and only if the following equation holds:
 \begin{equation*}T_{r}(\eta)\star T_{r}(\zeta)=T_{r}(L_{\ast}^{*}(T_{r}(\eta))\zeta-L_{\ast}^{*}(T_{\tau(r)}(\zeta))\eta),~~\forall~\eta,\zeta\in A^{*}.
\end{equation*}
\item
 $r$ is a solution of the equation $ D_{2}(r)=r_{23}\star r_{12}+r_{13}\ast r_{23}+r_{13}\ast r_{12}=0$
 in $(A,\ast,\circ)$
 if and only if the following equation holds:
 \begin{equation*}T_{r}(\eta)\ast T_{r}(\zeta)=T_{r}(R_{\ast}^{*}(T_{\tau(r)}(\zeta))\eta
+L_{\star}^{*}(T_{r}(\eta))\zeta),~~\forall~\eta,\zeta\in A^{*}.
\end{equation*}
\item
 $r$ is a solution of the
equation $D_{1}(r)= r_{12}\star r_{13}+r_{23}\ast r_{12}-r_{23}\ast r_{13}=0$ in $(A,\ast,\circ)$
if and only if the following equation holds:
 \begin{equation*}T_{\tau(r)}(\eta)\star T_{\tau(r)}(\zeta)=T_{\tau(r)}(L_{\ast}^{*}(T_{\tau(r)}(\zeta))\eta
 -L_{\ast}^{*}(T_{r}(\eta))\zeta),~~\forall~\eta,\zeta\in A^{*}.
\end{equation*}
 \end{enumerate}
\end{pro}
Recall that an $\mathcal O$-operator $T$ of an anti-Zinbiel
algebra $(A,\ast)$ associated to a representation $(V,l_{\ast},r_{\ast})$ is a linear map $T:V\longrightarrow A$ satisfying
$T(u)\ast T(v)=T (l_{\ast}(T(u))v+r_{\ast}(T(v))u)$ for all $u,v\in V$. An $\mathcal O$-operator $T$ of an anti-pre-Lie
algebra $(A,\circ)$ associated to a representation $(V,l_{\circ},r_{\circ})$ is a linear map $T:V\longrightarrow A$ satisfying
$T(u)\circ T(v)=T (l_{\circ}(T(u))v+r_{\circ}(T(v))u)$ for all $u,v\in V$. 

\begin{defi} Let $(A,\ast,\circ)$ be an anti-pre-Poisson algebra and $(V,l_{\ast},r_{\ast},l_{\circ},r_{\circ})$ be
 its representation.
  An $\mathcal O$-operator $T$ of $(A,\ast,\circ)$ associated to $(V,l_{\ast},r_{\ast},l_{\circ},r_{\circ})$
   is a linear map $T:V\longrightarrow A$ 
   which is both an $\mathcal O$-operator $T$ of $(A,\ast )$  associated to $(V,l_{\ast},r_{\ast} )$
   and an $\mathcal O$-operator $T$ of $(A,\circ )$  associated to $(V,l_{\circ},r_{\circ} )$, that is, the following equations hold:
   \begin{eqnarray}
 &&T(u)\ast T(v)=T (l_{\ast}(T(u))v+r_{\ast}(T(v))u),\\
 &&T(u)\circ T(v)=T (l_{\circ}(T(u))v+r_{\circ}(T(v))u),\;\forall~u,v\in V.
   \end{eqnarray}
\end{defi}

\begin{ex} \label{Pe1} Let $(A,\ast,\circ)$ be the 3-dimensional anti-pre-Poisson algebra given in Example \ref{Ae}
with a basis $\{e_1, e_2, e_3 \}$. Define a linear map $T:A\longrightarrow A$ by a matrix
 $\begin {bmatrix}
 t_{11}&0&0\\
t_{21}&t_{22}&0\\
t_{31}&t_{32}&t_{33}
\end {bmatrix}$
with respect to the basis $\{e_1, e_2, e_3 \}$, where $t_{ij}\in k~(i=1,2,3 )$. 
Then $T$ is an $\mathcal{O}$-operator of $(A,\ast,\circ)$ associated to the representation
$(A,L_{\ast},R_{\ast},L_{\circ},R_{\circ})$ if and only if the following equations hold:
\begin{equation*} t_{11}t_{21}=t_{21}t_{33},\ \ \ t_{11}^{2}=2t_{11}t_{33},\ \ t_{11}t_{22}=(t_{11}+t_{22})t_{33}.\end{equation*}
\end{ex}

\begin{ex} \label{Pe1} Let $(A,\ast,\circ)$ be the 3-dimensional anti-pre-Poisson algebra given in Example \ref{Ae}
with a basis $\{e_1, e_2, e_3 \}$. Assume that $\{e_{1}^{*}, e_{2}^{*}, e_{3}^{*} \}$ is the dual basis of $A^{*}$.
Define a linear map $T:A^{*}\longrightarrow A$ by a matrix
 $\begin {bmatrix}
 0&0&t_{13}\\
0&0&t_{23}\\
t_{31}&t_{32}&t_{33}
\end {bmatrix}$
with respect to the basis $\{e_{1}^{*}, e_{2}^{*}, e_{3}^{*} \}$, where $t_{ij}\in k~(i=1,2,3 )$. 
Then $T$ is an $\mathcal{O}$-operator of $(A,\ast,\circ)$ associated to the representation
$(A^{*},L_{\star}^{*},-R_{\ast}^{*},-\mathrm{ad}^{*},R_{\circ}^{*})$ if and only if the following equations hold:
\begin{equation*} t_{23}t_{13}=t_{13}t_{32}-2t_{23}t_{31},\ \ \ t_{13}^{2}=-t_{13}t_{31}.\end{equation*}
\end{ex}

Combining Proposition 3.23 and Proposition 3.24, we have the following conclusion.

\begin{thm} \label{Op1} Let $(A,\ast,\circ)$ be an anti-pre-Poisson algebra and
$r\in A\otimes A$ be skew-symmetric. Then the following conditions are equivalent:
\begin{enumerate}
		\item $r$ is a solution of the APP-YBE in $(A,\ast,\circ)$.
 \item $T_r$
is an $\mathcal O$-operator of the Poisson algebra $(A, \star, [ \ , \ ])$ associated to $(A^{*}, L_{\ast}^{*},-L_{\circ}^{*})$.
\item $T_r$
is an $\mathcal O$-operator of $(A,\ast,\circ)$ associated to 
$(A^{*},L_{\ast}^{*}+R_{\ast}^{*}, -R_{\ast}^{*}, R_{\circ}^{*}-L_{\circ}^{*},R_{\circ}^{*})$.
\end{enumerate}
\end{thm}

In the light of Theorem 6.19 \cite{20} and Theorem 2.26 \cite{13}, we have the following result.

\begin{thm} \label{Yo}
 Let $(A,\ast,\circ)$ be an anti-pre-Poisson algebra and $(V,l_{\ast},r_{\ast},l_{\circ},r_{\circ})$ be
  a representation of $(A,\ast,\circ)$. Suppose that $(V^{*},l_{\ast}^{*}+r_{\ast}^{*},-r_{\ast}^{*},r_{\circ}^{*}-l_{\circ}^{*},r^{*}_{\circ})$ 
 is the dual representation of $A$ given by Proposition \ref{Dr}. 
Let $\hat{A}=A\ltimes V^{*}$ and
 $T:V\longrightarrow A$ be a linear map which is identifies an element in $\hat{A}\otimes \hat{A}$ through
 ($Hom(V,A)\simeq A\otimes V^{*}\subseteq \hat{A}\otimes \hat{A}$).
  Then $r=T-\tau(T)$ is a skew-symmetric solution of
 the APP-YBE in the anti-pre-Poisson algebra $\hat{A}$ if and only if $T$ is an $\mathcal O$-operator
on $(A,\ast,\circ)$ associated to  $(V,l_{\ast},r_{\ast},l_{\circ},r_{\circ})$.
\end{thm}

\begin{ex} Let $(A=ke_1\oplus ke_2\oplus ke_3,\ast,\circ)$ be the 3-dimensional anti-pre-Poisson algebra with an 
$\mathcal{O}$-operator $T$ given in Example \ref{Pe1}.
Denote the dual basis of $A^{*}$ by $\{e_1^{*}, e_2^{*}, e_3^{*}  \}$. 
The semi-direct product $A\ltimes A^{*}$ of $(A,\ast,\circ)$ with respect to its representation
$(A^{*},L_{\ast}^{*}+R_{\ast}^{*}, -R_{\ast}^{*}, R_{\circ}^{*}-L_{\circ}^{*},R_{\circ}^{*})$
 is an anti-pre-Poisson algebra with the multiplications $(\ast,\circ)$ defined by 
 \begin{align*}&
 (x+u)\ast (y+v)=x\ast y+(L_{\ast}^{*}+R_{\ast}^{*})(x)v-R_{\ast}^{*}(y)u,\\&
		 (x+u)\circ (y+v)=x\circ y+(R_{\circ}^{*}-L_{\circ}^{*})(x)v+R_{\circ}^{*}(y)u, \ \ ~\forall~x,y\in A,u,v\in A^{*}.
	\end{align*}
 Explicitly, the non-trivial multiplications $(\ast,\circ)$ are given by 
   \begin{flalign*}
&e_1\ast e_1=e_3, \ \ \ e_1\ast e_3^{*}= 2e_1^{*}, \ \ \ e_3^{*}\ast e_1= e_1^{*},\\&
e_1\circ e_2=e_3, \ \ \ e_1\circ e_{3}^{*}= e_2^{*}, \ \ \ e_2\circ e_{3}^{*}=e_3^{*}\circ e_2= -e_1^{*}.
\end{flalign*}
 In the light of Theorem \ref{Yo}, we know that
 $r=\sum_{i =1}^{3}T(e_i)\otimes e_i^{*}-e_i^{*}\otimes T(e_i)$
 is a skew-symmetric solution of the APP-YBE in the anti-pre-Poisson algebra $ A\ltimes A^{*} $. 
 According to Corollary \ref{Bc2}, 
$ (A\ltimes A^{*}, \Delta,\delta)$ is an anti-pre-Poisson bialgebra with the linear maps
$\Delta,\delta:A\ltimes A^{*}\longrightarrow (A\ltimes A^{*})\otimes (A\ltimes A^{*})$ defined respectively by
 \begin{align*}&\Delta(x)=-(I\otimes L_{\star}(x)+L_{\ast}(x)\otimes I)r,\\&
 \delta(x)=(L_{\circ}(x)\otimes I-I\otimes \mathrm{ad}(x))r,~~\forall~x\in A\ltimes A^{*}.
\end{align*}
Explicitly, the comultiplications $(\Delta,\delta)$ are given as follows (only non-zero operations are listed): 
 \begin{align*}&
\Delta(e_1)=2(t_{11}-t_{33})e_1^{*}\otimes e_3+(t_{33}-t_{11})e_3\otimes e_1^{*}, \ \ \ \Delta(e_{3}^{*})=t_{11}e_1^{*}\otimes e_1^{*},
\\&
\delta(e_1)=t_{21}e_3\otimes e_1^{*}+(t_{22}-t_{33})e_3\otimes e_2^{*} -e_2^{*}\otimes (t_{11}e_1+t_{21}e_2 +(t_{31}-t_{22})e_3)
+t_{21}e_1^{*}\otimes e_3,\\&
  \delta(e_2)=-(t_{11}+t_{33})e_1^{*}\otimes e_3,\ \ \ \delta(e_3^{*})=-t_{21}e_1^{*}\otimes e_1^{*}-(t_{11}+t_{22})e_1^{*}\otimes e_2^{*}.
\end{align*}
\end{ex}

%%%%%%%%%%%%%%%%%%%%%%%%%%%%%%%%%%%%%%%%%%%%%%%%%%%%%%%%%%%%%%%%%%%%%%%%%%%%%%%%%%%%%%%%%%%%%%%%%%%%%%%%%%
\section{Quasi-triangular anti-pre-Poisson bialgebras and factorizable anti-pre-Poisson bialgebras}

By Theorem~\ref{Ac5}, a skew-symmetric solution of the APP-YBE yields a (coboundary) anti-pre-Poisson bialgebra. 
Here, we explore how solutions of the APP-YBE lead to anti-pre-Poisson bialgebras in general, 
demonstrating that such structures need not be skew-symmetric.

\subsection{Quasi-triangular anti-pre-Poisson bialgebras }

\begin{defi} \label{In}
 Let $(A,\ast,\circ)$ be an anti-pre-Poisson algebra and $r\in A\otimes A$. Then $r$ is called {\bf invariant} if 
 \begin{align}&\label{IE3}(I\otimes L_{\star}(x)+L_{\ast}(x)\otimes I)r=0,
\\&\label{IE4}(L_{\circ}(x)\otimes I-I\otimes \mathrm{ad}(x))r=0,~\forall~x\in A.
\end{align}
\end{defi}

\begin{lem} \label{In1}
 Let $(A,\ast,\circ)$ be an anti-pre-Poisson algebra and $r\in A\otimes A$. Then $r$ is invariant if and only if
 \begin{align}&\label{IE5}
L_{\star}(x) T_{r}(\zeta)=T_{r}(L_{\ast}^{*}(x)\zeta),
\\&\label{IE6}\mathrm{ad}(x) T_{r}(\zeta)=-T_{r}(L_{\circ}^{*}(x)\zeta),~~\forall~x\in A,\zeta\in A^{*}.
\end{align}
Moreover, Eqs.~(\ref{IE5})-(\ref{IE6}) hold if and only if the following equations hold:
\begin{align}&\label{IE55}
L_{\star}^{*} (T_{r}(\zeta))\eta=-R_{\ast}^{*}(T_{\tau(r)}(\eta))\zeta),
\\&\label{IE66}\mathrm{ad}^{*} (T_{r}(\zeta))\eta=-R_{\circ}^{*}(T_{\tau(r)}(\eta))\zeta,~~\forall~x\in A,\zeta\in A^{*}.
\end{align}
\end{lem}
\begin{proof} 
For all $~x\in A,\zeta,\eta\in A^{*}$, we have
\begin{small}
\begin{align*}\langle(I\otimes L_{\star}(x)+L_{\ast}(x)\otimes I)r,\zeta\otimes \eta\rangle
=&-\langle r,L_{\ast}^{*}(x)\zeta\otimes \eta\rangle-\langle r,\zeta\otimes L_{\star}^{*}(x)\eta\rangle
\\=&-\langle T_{r}(L_{\ast}^{*}(x))\zeta,\eta\rangle-\langle T_{r}(\zeta), L_{\star}^{*}(x)\eta\rangle
\\=&\langle L_{\star}(x)T_{r}(\zeta)-T_{r}(L_{\ast}^{*}(x))\zeta,\eta\rangle,\end{align*}
\begin{align*} \langle(I\otimes \mathrm{ad}(x)- L_{\circ}(x)\otimes I)r,\zeta\otimes \eta\rangle
=&\langle r,L_{\circ}^{*}(x)\zeta\otimes \eta\rangle-\langle r,\zeta\otimes  \mathrm{ad}^{*}(x)\eta\rangle
\\=&\langle T_{r}(L_{\circ}^{*}(x))\zeta,\eta\rangle-\langle T_{r}(\zeta), \mathrm{ad}^{*}(x)\eta\rangle
\\=&\langle T_{r}(L_{\circ}^{*}(x))\zeta+\mathrm{ad}(x) T_{r}(\zeta), \eta\rangle.
\end{align*}\end{small}
Thus, Eqs.~(\ref{IE3})-(\ref{IE4}) hold if and only if Eqs.~(\ref{IE5})-(\ref{IE6}) hold. Analogously,
Eqs.~(\ref{IE5})-(\ref{IE6}) hold if and only if Eqs.~(\ref{IE55})-(\ref{IE66}) hold. 
\end{proof}

\begin{pro} \label{Si}
 Let $(A,\ast,\circ)$ be an anti-pre-Poisson algebra and $r\in A\otimes A$. Then the following
 conditions are equivalent:
  \begin{enumerate}
\item $r+\tau(r)$ is invariant.
\item The following equations hold:
\begin{align}&\label{IE7}L_{\star}^{*} (T_{r+\tau(r)}(\zeta))\eta=-R_{\ast}^{*}(T_{r+\tau(r)}(\eta))\zeta)
, \ \ \ \mathrm{ad}^{*} (T_{r+\tau(r)}(\zeta))\eta=-R_{\circ}^{*}(T_{r+\tau(r)}(\eta))\zeta.\end{align} 
\item The following equations hold:
\begin{align}&\label{IE8}L_{\star}(x) T_{r+\tau(r)}(\zeta)=-T_{r+\tau(r)}(L_{\ast}^{*}(x)\zeta)
,\ \ \ \mathrm{ad}(x) T_{r+\tau(r)}(\zeta)=-T_{r+\tau(r)}(L_{\circ}^{*}(x)\zeta).\end{align}
\item The following equations hold:
\begin{align}&\label{LR1} 
T_{r+\tau(r)}(L_{\star}^{*}(x)\zeta)= -L_{\ast}(x)T_{r+\tau(r)}(\zeta),\ \ \
 T_{r+\tau(r)}(\mathrm{ad}^{*}(x)\zeta)=-L_{\circ}(x)T_{r+\tau(r)}(\zeta),
\end{align}
\end{enumerate}
for all $x\in A,~\zeta,\eta\in A^{*}$
 \end{pro}
  
   \begin{proof} In view of Lemma \ref{In1}, $r+\tau(r)$ is invariant if and only if Eq.~(\ref{IE8}) holds.
 Since $T_{r+\tau(r)}=T_{r+\tau(r)}^{*}$, Eq.~(\ref{IE8}) holds if and only if Eq.~ (\ref{LR1}) holds.
By Eqs.~(\ref{IE55})-(\ref{IE66}), $r+\tau(r)$ is invariant if and only if Eq.~(\ref{IE7}) holds. 
The proof is completed. 
\end{proof}
  
By Eqs.~(\ref{IE8}) and (\ref{LR1}), we have 
 \begin{align}&\label{LR2} 
T_{r+\tau(r)}(R_{\ast}^{*}(x)\zeta)= R_{\ast}(x)T_{r+\tau(r)}(\zeta),\ \ \
 T_{r+\tau(r)}(R_{\circ}^{*}(x)\zeta)=R_{\circ}(x)T_{r+\tau(r)}(\zeta),
 \\&\label{LR3} 
 L_{\ast}^{*} (T_{r+\tau(r)}(\zeta))\eta=L_{\ast}^{*}(T_{r+\tau(r)}(\eta))\zeta)
, \ \ \ L_{\circ}^{*} (T_{r+\tau(r)}(\zeta))\eta=-L_{\circ}^{*}(T_{r+\tau(r)}(\eta))\zeta
\end{align}
  
 \begin{thm}\label{Ys3} Let $(A,\circ)$ be an anti-pre-Lie
 algebra and $r=\sum_{i}a_i\otimes b_i\in A\otimes A$. 
Assume that $( L_{\circ}(x)\otimes I-I\otimes  \mathrm{ad}(x))(r+\tau(r))=0$.
Then the following conditions are equivalent:
\begin{enumerate}
		\item $r$ is a solution of the APL-YBE $S(r)=r_{12}\circ r_{13}+r_{12}\circ r_{23}-[r_{13},r_{23}]=0$.
\item $r$ is a solution of the equation $S_{1}(r)=[r_{13},r_{12}]-r_{23}\circ r_{12}-r_{23}\circ r_{13}=0$.
\item $r$ is a solution of the equations $S_{2}(r)=[r_{12},r_{23}]+r_{13}\circ r_{12}-r_{13}\circ r_{23}=0$. 
\item $r$ is a solution of the equation $S(\tau(r))=r_{21}\circ r_{31}+r_{21}\circ r_{32}+[r_{32}, r_{31}]=0$.
\end{enumerate}
\end{thm}

\begin{proof} Since 
$( L_{\circ}(x)\otimes I-I\otimes  \mathrm{ad}(x))(r+\tau(r))=0$,
$( R_{\circ}(x)\otimes I+I\otimes  R_{\circ}(x))(r+\tau(r))=0$.
 Note that 
 \begin{small}
\begin{align*}S_{1}(r)=&\sigma_{123}S(r)-r_{23}\circ (r_{12}+r_{21})-r_{23}\circ (r_{13}+r_{31})
+[r_{13},r_{12}+r_{21}]-[r_{13}+r_{31},r_{21}]
\\=&\sigma_{123}S(r)-\sum_{i}(I\otimes L_{\circ}(a_i)\otimes I)[(r+\tau(r))\otimes b_i]
+(I\otimes I\otimes L_{\circ}(b_i))(\tau\otimes I)[a_i\otimes(r+\tau(r)) ]\\&
-(\mathrm{ad}(a_i)\otimes I\otimes I)[(r+\tau(r))\otimes b_i]
-(\mathrm{ad}(b_i)\otimes I\otimes I)(\tau\otimes I)[a_i\otimes(r+\tau(r)) ],
\end{align*}
\begin{align*}S(\tau(r))=&S(r)
-r_{12}\circ (r_{13}+r_{31})-[r_{13}+r_{31},r_{32}]
-r_{12}\circ (r_{23}+r_{32})+[r_{13},r_{23}+r_{32}]
\\&+(r_{12}+r_{21})\circ r_{31}+(r_{12}+r_{21})\circ r_{32}
\\=&S(r)-\sum_{i}( L_{\circ}(a_i)\otimes I\otimes I-I\otimes I\otimes \mathrm{ad}(a_i))(\tau\otimes I)[b_i\otimes(r+\tau(r)) ]
\\&+(I\otimes L_{\circ}(b_i)\otimes  I-I\otimes I\otimes \mathrm{ad}(b_i))[a_i\otimes(r+\tau(r)) ]
\\&-(I\otimes R_{\circ}(b_i)\otimes  I+ R_{\circ}(b_i)\otimes I\otimes I)[(r+\tau(r))\otimes a_i],
\end{align*}
\begin{align*}S_{2}(r)=&\sigma_{23}S(r)-r_{13}\circ (r_{23}+r_{32})
+[r_{12},r_{23}+r_{32}]
\\=&\sigma_{23}S(r)-\sum_{i}(I\otimes I\otimes L_{\circ}(b_i)-I\otimes \mathrm{ad}(b_i)\otimes I)[a_i\otimes(r+\tau(r)) ].
\end{align*}\end{small}
Thus, Item (a) $\Longleftrightarrow $ Item (b), Item (a) $\Longleftrightarrow $ Item (c),
 Item (a) $\Longleftrightarrow $ Item (d).
\end{proof}
 \begin{thm}\cite{19}\label{Ys4} Let $(A,\ast)$ be an anti-Zinbiel
 algebra and $r=\sum_{i}a_i\otimes b_i\in A\otimes A$. 
Assume that $(I\otimes L_{\star}(x)+L_{\ast}(x)\otimes I)(r+\tau(r))=0$. Then the following conditions are equivalent:
\begin{enumerate}
		\item $r$ is a solution of the AZ-YBE $D(r)=r_{12}\ast r_{23}-r_{12}\ast r_{13}+r_{13}\star r_{23}=0$.
\item $r$ is a solution of the equation $D_{1}(r)=r_{12}\star r_{13}+r_{23}\ast r_{12}-r_{23}\ast r_{13}=0$.
\item $r$ is a solution of the equations $D_{2}(r)=r_{13}\ast r_{23}+r_{12}\star r_{23}+r_{13}\ast r_{12}=0$. 
\item $r$ is a solution of the equation $D(\tau(r))=r_{21}\ast r_{32}-r_{21}\ast r_{31}+r_{31}\star r_{32}=0$.
\end{enumerate}
\end{thm}

\begin{thm} \label{Qs} Let $(A,\ast,\circ)$ be an anti-pre-Poisson algebra and $r=\sum_{i}a_i\otimes b_i\in A\otimes A$. 
 Assume that
$\Delta_r,\delta_r$ are given by Eqs.~(\ref{CB1}) and (\ref{CB2}). If $r$ is a solution of the 
APP-YBE in $(A,\ast,\circ)$ and $r+\tau(r)$ is invariant.
Then $(A,\ast,\circ,\Delta_r,\delta_r)$ is an anti-pre-Poisson bialgebra.
\end{thm}
\begin{proof}
Since $r$ is a solution of the 
APP-YBE in $(A,\ast,\circ)$ and $r+\tau(r)$ is invariant, $S(r)=0,~D(r)=0$ and
Eqs.~(\ref{IE7})-(\ref{LR2}) hold. By Theorem 3.5 \cite{19}, $(A,\ast,\Delta_r)$ is an anti-Zinbiel bialgebra.
Using Eqs.~(\ref{IE8})-(\ref{LR2}) and (\ref{r3}), we have for all $\zeta,\eta,\theta\in A^{*}$,
\begin{small}
\begin{align*}&
\langle(I\otimes L_{\circ} (x\ast a_i)-L_{\circ} (x\ast a_i)\otimes I-L_{\circ} (x)R_{\ast} (a_i)\otimes I
+ R_{\circ} (a_i)\otimes L_{\ast}(x))(r+\tau(r)),
\zeta\otimes\eta\rangle
\\=&\langle r+\tau(r),L_{\circ}^{*} (x\ast a_i)\zeta\otimes \eta-\zeta\otimes L_{\circ}^{*} (x\ast a_i)\eta
-R_{\ast}^{*} (a_i)L_{\circ}^{*} (x)\zeta\otimes \eta+ R_{\circ}^{*} (a_i)\zeta\otimes L_{\ast}^{*}(x)\eta\rangle
\\=&\langle T_{r+\tau(r)}(L_{\circ}^{*} (x\ast a_i)\zeta), \eta\rangle
-\langle T_{r+\tau(r)}(\zeta), L_{\circ}^{*} (x\ast a_i)\eta\rangle
\\&-\langle T_{r+\tau(r)}(R_{\ast}^{*} (a_i)L_{\circ}^{*} (x)\zeta),\eta\rangle
+\langle T_{r+\tau(r)}(R_{\circ}^{*} (a_i)\zeta), L_{\ast}^{*}(x)\eta\rangle
\\=&\langle- \mathrm{ad}(x\ast a_i)T_{r+\tau(r)}(\zeta), \eta\rangle
+\langle L_{\circ} (x\ast a_i)T_{r+\tau(r)}(\zeta), \eta\rangle
\\&+\langle R_{\ast}(a_i)\mathrm{ad}(x)T_{r+\tau(r)}(\zeta),\eta\rangle
-\langle R_{\circ} (a_i)L_{\ast}(x)T_{r+\tau(r)}(\zeta), \eta\rangle
\\=&0,\end{align*}
\end{small}
which follows that
\begin{small}
\begin{align*}&
(I\otimes L_{\circ} (x\ast a_i)\otimes I-L_{\circ} (x\ast a_i)\otimes I\otimes I)[(r+\tau(r))\otimes b_i)]
\\&-(L_{\circ} (x)\otimes I\otimes I)( R_{\ast} (a_i)\otimes I\otimes I)[(r+\tau(r))\otimes b_i]
\\&+(I\otimes L_{\ast}(x)\otimes I)(  R_{\circ} (a_i)\otimes I\otimes I)[(r+\tau(r))\otimes b_i]=0.
\end{align*}
\end{small}
Note that
\begin{small}
\begin{align*}&
[r_{23}, r_{12}]-r_{13}\circ r_{12}+r_{23}\circ r_{21}-[r_{13}, r_{21}]+[r_{13}, r_{23}])	
\\=&-S(r)+[r_{12}+r_{21},r_{13}]+r_{23}\circ (r_{12}+r_{21})
\\=&-S(r)+\sum_{i}(I\otimes  L_{\circ} (a_i)\otimes I-\mathrm{ad}(a_i)\otimes I\otimes I)[(r+\tau(r))\otimes b_i],\\
&r_{13}\star r_{23}+r_{12}\ast r_{23}+r_{21}\ast r_{13}=D(r)+(r_{12}+r_{21})\ast r_{13}
\\=&D(r)+\sum_{i}( R_{\ast} (a_i)\otimes I\otimes I)[(r+\tau(r))\otimes b_i],
\\&
r_{21}\circ r_{13} -r_{12}\circ r_{23}+[r_{13}, r_{23}]=-S(r)+(r_{12}+r_{21})\circ r_{13}
\\=&-S(r)+\sum_{i}(  R_{\circ} (a_i)\otimes I\otimes I)[(r+\tau(r))\otimes b_i].
\end{align*}
\end{small}
Therefore,
\begin{small}
\begin{align*}&
(I\otimes I\otimes L_{\star} (x))([r_{23}, r_{12}]-r_{13}\circ r_{12}+r_{23}\circ r_{21}-[r_{13}, r_{21}]+[r_{13}, r_{23}])	
\\&-(L_{\circ} (x\ast a_i)\otimes I\otimes I)((r+\tau(r))\otimes b_i)
+(I\otimes L_{\circ} (x\ast a_i)\otimes I)((r+\tau(r))\otimes b_i)
\\&-(L_{\circ} (x)\otimes I\otimes I)(r_{13}\star r_{23}+r_{12}\ast r_{23}+r_{21}\ast r_{13})
\\&+(I\otimes L_{\ast}(x)\otimes I)(r_{21}\circ r_{13} -r_{12}\circ r_{23}+[r_{13}, r_{23}])=0,
\end{align*}
\end{small}
that is, Eq.~(\ref{CA1}) hold. Analogously, Eqs.~(\ref{CA2})-(\ref{CA5}) and (\ref{Pa5})-(\ref{Pa7}) hold. 
Combining Theorem \ref{Bc1}, we complete the proof.
\end{proof}

\begin{pro} \label{Ys5}
Let $(A,\ast,\circ,\Delta_r,\delta_r)$ be an anti-pre-Poisson bialgebra and $r\in A\otimes A$, 
where the comultiplications
$\Delta_r,\delta_r$ are defined by Eqs.~(\ref{CB1}) and (\ref{CB2}).
 Then the anti-pre-Poisson algebra structure $(\ast_r,\circ_r)$ on $A^{*}$ is defined
 by 
 \begin{small}
\begin{align}&\label{Ic1}\zeta\ast_{r}\eta=R_{\ast}^{*}(T_{\tau(r)}(\eta))\zeta+L_{\star}^{*}(T_{r}(\zeta))\eta,
\\&\label{Ic2}\zeta\circ_{r}\eta=-\mathrm{ad}^{*}(T_{r}(\zeta))\eta-R_{\circ}^{*}(T_{\tau(r)}(\eta))\zeta.
\end{align}
\end{small}
And the associated Poisson algebra structure $(\star_r,[ \ , \ ]_{r})$ on $A^{*}$ is
given by 
\begin{equation*}\zeta\star_{r}\eta=\zeta\ast_{r}\eta+\eta\ast_{r}\zeta
, \ \ \ [ \zeta , \eta ]_{r}=\zeta\circ_{r}\eta-\eta\circ_{r}\zeta,
\end{equation*}
where $\mathrm{ad}=L_{\circ}-R_{\circ},~ L_{\star}=R_{\ast}+L_{\ast}$.
\end{pro}
\begin{proof} For all $\zeta,\eta\in A^{*}$ and $x\in A$,
\begin{small}
\begin{align*}\langle \zeta\ast_{r}\eta,x\rangle
&=\langle \zeta\otimes\eta,\Delta_{r}(x)\rangle
=-\langle \zeta\otimes\eta,(I\otimes L_{\star}(x)+L_{\ast}(x)\otimes I)r\rangle
\\&=\langle   L_{\ast}^{*}(x)\zeta\otimes\eta,r\rangle
+\langle  \zeta \otimes L_{\star}^{*}(x)\eta,r\rangle
=\langle T_{\tau(r)}(\eta),L_{\ast}^{*}(x)\zeta\rangle
+\langle T_{r}(\zeta),L_{\star}^{*}(x)\eta\rangle
\\&=-\langle x\ast T_{\tau(r)}(\eta),\zeta\rangle
-\langle x\star T_{r}(\zeta),\eta\rangle
=\langle R_{\ast}^{*}(T_{\tau(r)}(\eta))\zeta,x\rangle+\langle L_{\star}^{*}(T_{r}(\zeta))\eta,x\rangle,
\end{align*}\end{small}
which indicates that Eq.~(\ref{Ic1}) holds.
By the same token, Eq.~(\ref{Ic2}) holds.
\end{proof}

\begin{defi} \label{Qt1}
 Let $(A,\ast,\circ)$ be an anti-pre-Poisson algebra and $r\in A\otimes A$. If $r$ is a solution of the APP-YBE in $(A,\ast,\circ)$
 and $r+\tau(r)$ is invariant, then the anti-pre-Poisson
  bialgebra $(A, \ast,\circ,\Delta_{r},\delta_{r})$ induced by $r$ is called a {\bf quasi-triangular} anti-pre-Poisson
 bialgebra. In particular, if $r$ is skew-symmetric, $(A, \ast,\circ,\Delta_{r},\delta_{r})$ 
is called a {\bf triangular} anti-pre-Poisson bialgebra, where $\Delta_{r}$ and $\delta_{r}$ 
are given by Eqs.~(\ref{CB1}) and (\ref{CB2}) respectively.
\end{defi}

\begin{thm}\label{QB0}  Let $(A,\ast,\circ)$ be an anti-pre-Poisson algebra and
$r\in A\otimes A$. Suppose that $r+\tau(r)$ is invariant. Then the following conditions are equivalent:
 \begin{enumerate}
\item $r$ is a solution of the APP-YBE in $(A,\ast,\circ)$.
 \item $(A^{*},\star_r, [ \ , \ ]_r)$ is a Poisson algebra and the linear maps
$T_{r},-T_{\tau(r)}$ are both Poisson algebra
 homomorphisms from $(A^{*},\star_r, [ \ , \ ]_r)$ to $(A,\star, [ \ , \ ])$.
 \item $(A^{*},\ast_r,\circ_r)$ is an anti-pre-Poisson algebra and the linear maps
$T_{r},-T_{\tau(r)}$ are both anti-pre-Poisson algebra
 homomorphisms from $(A^{*},\ast_r,\circ_r)$ to $(A,\ast,\circ)$.
 \end{enumerate}
\end{thm}

\begin{proof} 
Using Eqs.~(\ref{IE7}) and (\ref{Ic1}), we have
\begin{small}
\begin{align*}&\zeta\star_{r}\eta=R_{\ast}^{*}(T_{\tau(r)}(\eta))\zeta+L_{\star}^{*}(T_{r}(\zeta))\eta+
R_{\ast}^{*}(T_{\tau(r)}(\zeta))\eta+L_{\star}^{*}(T_{r}(\eta))\zeta,
\\=&R_{\ast}^{*}(T_{r+\tau(r)}(\eta))\zeta+L_{\ast}^{*}(T_{r}(\eta))\zeta
+R_{\ast}^{*}(T_{r+\tau(r)}(\zeta))\eta+L_{\ast}^{*}(T_{r}(\zeta))\eta
\\=&-L_{\star}^{*}(T_{r+\tau(r)}(\zeta))\eta+L_{\ast}^{*}(T_{r}(\eta))\zeta+
R_{\ast}^{*}(T_{r+\tau(r)}(\zeta))\eta+L_{\ast}^{*}(T_{r}(\zeta))\eta+L_{\ast}^{*}(T_{r}(\zeta))\eta
\\=& L_{\ast}^{*}(T_{r}(\eta))\zeta-L_{\ast}^{*}(T_{\tau(r)}(\zeta))\eta.
\end{align*}\end{small}
Analogously,
\begin{small}
\begin{align*}
[ \zeta , \eta ]_{r}=L_{\circ}^{*}(T_{\tau(r)}(\zeta))\eta+L_{\circ}^{*}(T_{r}(\eta))\zeta.
\end{align*}\end{small}
Combining Proposition \ref{Ys1}, Proposition \ref{Ys2}, Theorem \ref{Ys3}, Theorem \ref{Ys4} and Proposition \ref{Ys5},
 we get the conclusion.
\end{proof}

\begin{cor} \label{QB} Let $(A, \ast,\circ,\Delta_{r},\delta_{r})$ be a quasi-triangular anti-pre-Poisson bialgebra.
 Then  \begin{enumerate}
\item $T_{r},-T_{\tau(r)}$ are both Poisson algebra
 homomorphisms from $(A^{*},\star_r, [ \ , \ ]_r)$ to $(A,\star, [ \ , \ ])$.
 \item $T_{r},-T_{\tau(r)}$ are anti-pre-Poisson algebra
 homomorphisms from $(A^{*},\ast_r,\circ_r)$ to $(A,\ast,\circ)$.
 \end{enumerate}
\end{cor}

\begin{proof} 
It follows from Theorem \ref{QB0} and Definition \ref{Qt1}.
\end{proof}

%%%%%%%%%%%%%%%%%%%%%%%%%%%%%%%%%%%%%%%%%%%%%%%%%%%%%%%%%%%%%%%%%%%%%%%%%%%%%%%%%%%%%%%%%%%%%%%%%%%%%%%%%%
\subsection{Factorizable anti-pre-Poisson bialgebras}
In this section, we introduce the notion of factorizable anti-pre-Poisson
bialgebras, which is a special class of quasi-triangular anti-pre-Poisson bialgebras.
We show that the double of an anti-pre-Poisson bialgebra
is naturally a factorizable anti-pre-Poisson bialgebra.

\begin{defi} \label{Qt} A quasi-triangular anti-pre-Poisson
 bialgebra $(A, \ast,\circ, \Delta_{r}, \delta_{r})$ is
called factorizable if  
 the linear map $T_{r+\tau(r)}:A^{*}\longrightarrow A$ given by Eq.~(\ref{YE7}) is a linear isomorphism of vector spaces, 
where $\Delta_{r}, \delta_{r} $
are defined by Eqs.~(\ref{CB1}) and (\ref{CB2}) respectively.\end{defi}

\begin{pro} \label{Qf} Let $(A, \ast,\circ, \Delta_{r}, \delta_{r})$ be a quasi-triangular
 (factorizable) anti-pre-Poisson bialgebra. Then $(A, \ast,\circ, \Delta_{\tau(r)}, \delta_{\tau(r)})$
 is also a quasi-triangular
 (factorizable) anti-pre-Poisson bialgebra.\end{pro}
\begin{proof}
On the basis of Definition \ref{Qt} and Items (d) of Theorems \ref{Ys3} and \ref{Ys4}, we conclude that the statement holds.
\end{proof}

\begin{pro}
Let $(A, \ast,\circ, \Delta_{r}, \delta_{r})$ be a factorizable anti-pre-Poisson bialgebra.
Then $\mathrm{Im}(T_{r}\oplus T_{\tau(r)})$ is an anti-pre-Poisson subalgebra of the direct sum anti-pre-Poisson
algebra $A\oplus A$,
which is isomorphic to the anti-pre-Poisson algebra $(A^{*},\ast_{r},\circ_{r})$. Furthermore, any $x\in A$ has a unique
decomposition $x=x_1-x_2$, where $(x_1,x_2)\in \mathrm{Im}(T_{r}\oplus T_{\tau(r)})$ and
\begin{equation*}T_{r}\oplus T_{\tau(r)}:A^{*}\longrightarrow A\oplus A,~~(T_{r}\oplus T_{\tau(r)})(\zeta)=(T_{r}(\zeta),
 -T_{\tau(r)}(\zeta)),~\forall~\zeta\in A^{*}. \end{equation*}
\end{pro}

\begin{proof} In view of Definition \ref{Qt} and Corollary \ref{QB}, $T_{r}, -T_{\tau(r)}$ are anti-pre-Poisson algebra 
homomorphisms and $T_{r+\tau(r)}$ is a linear isomorphism. It follows that
$T_{r}\oplus T_{\tau(r)}$ is an anti-pre-Poisson algebra homomorphism and $\mathrm{Ker}(T_{r}\oplus T_{\tau(r)})=0$.
Therefore, $\mathrm{Im}(T_{r}\oplus T_{\tau(r)})$ is isomorphic to $(A^{*},\ast_{r},\circ_{r})$ as anti-pre-Poisson algebras.
For all $x\in A$, 
\begin{equation*}x=T_{r+\tau(r)}T_{r+\tau(r)}^{-1}(x)=T_{r}T_{r+\tau(r)}^{-1}(x)+T_{\tau(r)}T_{r+\tau(r)}^{-1}(x)=x_1-x_2
,\end{equation*}
where $x_1=T_{r}T_{r+\tau(r)}^{-1}(x)\in \mathrm{Im}(T_{r}),~
x_2=-T_{\tau(r)}T_{r+\tau(r)}^{-1}(x)\in \mathrm{Im}(-T_{\tau(r)})$.
The proof is finished.
\end{proof}

\begin{pro}
 Let $(A, \ast_{A},\circ_{A}, \Delta_{r}, \delta_{r})$
 be a factorizable anti-pre-Poisson bialgebra. Then the double anti-pre-Poisson algebra $(D=A\oplus A^{*},\ast_D,\circ_D)$
is isomorphic to the direct sum $A\oplus A$ of anti-pre-Poisson algebras, 
where $\ast_D,\circ_D$ are given by Eqs.~(\ref{Db1})-(\ref{Db2}) respectively.
\end{pro}

\begin{proof} Since $(A, \ast,\circ, \Delta_{r}, \delta_{r})$
 is a factorizable anti-pre-Poisson bialgebra,
$T_{r+\tau(r)}$ is a linear isomorphism and $r+\tau(r)$ is invariant.
Define $\varphi:A\oplus A^{*}\longrightarrow A\oplus A$ by
\begin{equation}\label{FD1}\varphi(x,\zeta)=(x+T_{r}(\zeta),x-T_{\tau(r)}(\zeta)),~\forall~(x,\zeta)\in A\oplus A^{*}.\end{equation}
Then $\varphi$ is a bijection.
Using Eqs.~(\ref{Ic2}) and (\ref{IE7}), we have
\begin{small}
\begin{align*}&\langle\eta,-T_{r}(\mathrm{ad}_{A}^{*}(x)\zeta)+R_{\circ_{A^*}}^{*}(\zeta)x\rangle
\\=&\langle \zeta,[x,T_{\tau(r)}(\eta)]\rangle-\langle x,\eta\circ\zeta\rangle
\\=&\langle x,\mathrm{ad}_{A}^{*}(T_{\tau(r)}(\eta))\zeta\rangle+\langle x,\mathrm{ad}^{*}(T_{r}(\eta))\zeta+R_{\circ_A}^{*}(T_{\tau(r)}(\zeta))\eta
\rangle
\\=&\langle x,\mathrm{ad}_{A}^{*}(T_{r+\tau(r)}(\eta))\zeta+R_{\circ_A}^{*}(T_{\tau(r)}(\zeta))\eta\rangle
\\=&\langle x,-R_{\circ_A}^{*}(T_{r+\tau(r)}(\zeta))\eta+R_{\circ_A}^{*}(T_{\tau(r)}(\zeta))\eta\rangle
\\=&\langle \eta,x\circ_{A} T_{r}(\zeta)\rangle.
\end{align*}\end{small}
Thus, \begin{equation}\label{FD2}x\circ_{A} T_{r}(\zeta)=R_{\circ_{A^*}}^{*}(\zeta)x-T_{r}(\mathrm{ad}_{A}^{*}(x)\zeta).\end{equation}
By the same token,
\begin{equation}\label{FD3}-x\circ_{A} T_{\tau(r)}(\zeta)=R_{\circ_{A^*}}^{*}(\zeta)x+T_{\tau(r)}(\mathrm{ad}_{A}^{*}(x)\zeta).\end{equation}
By Eqs.~(\ref{Db2}) and (\ref{FD2})-(\ref{FD2}), we get 
\begin{small}
\begin{align*}&\varphi(x\circ_D \zeta)=\varphi(R_{\circ_A^{*}}^{*}(\zeta)x
-\mathrm{ad}_{A^*}^{*}(x)\zeta)\\
=&(R_{\circ_A^{*}}^{*}(\zeta)x-T_{r}(\mathrm{ad}_{A}^{*}(x)\zeta),R_{\circ_A^{*}}^{*}(\zeta)x+T_{\tau(r)}(\mathrm{ad}_{A}^{*}(x)\zeta)
\\=&(x\circ_A T_{r}(\zeta),-x\circ_A T_{\tau(r)}(\zeta))
=(x,x)\circ (T_{r}(\zeta),-T_{\tau(r)}(\zeta))
\\=&\varphi(x)\circ_A \varphi(\zeta).
\end{align*}\end{small}
Analogously, $\varphi(\zeta\circ_D x)=\varphi(\zeta)\circ_A \varphi(x) $. Thus,
$\varphi((x,\zeta)\circ_D (y,\eta))=\varphi(x,\zeta)\circ \varphi(y,\eta) $ for all
$(x,\zeta), (y,\eta)\in A\oplus A^{*}$.
In the light of Proposition 4.14 \cite{19}, $\varphi((x,\zeta)\ast_D (y,\eta))=\varphi(x,\zeta)\ast \varphi(y,\eta) $.
In all, $\varphi$ is an isomorphism of anti-pre-Poisson algebras.
The proof is completed.
\end{proof}

\begin{thm} \label{Ft}
Let $(A, \ast_{A},\circ_{A}, \Delta_{r}, \delta_{r})$ be an anti-pre-Poisson
 bialgebra. Assume that $\{e_1,...,e_n\}$ is a basis of $A$ and $\{e^{*}_1,...,e^{*}_n\}$ is the dual basis.
 Let \begin{equation*}r=\sum_{i=1}^{n}e_{i}\otimes e_{i}^{*}\in A\otimes A^{*}\subseteq D\otimes D.\end{equation*}
Then $(D,\ast_D,\circ_D, \Delta_{r},\delta_{r})$ 
 is a factorizable anti-pre-Poisson bialgebra with 
$\Delta_{r},\delta_{r}$ given by Eqs.~(\ref{CB1}) and (\ref{CB2}) respectively.
\end{thm}

\begin{proof} In view of Eqs.~(\ref{Db1}) and (\ref{Db2}), we get
 \begin{small}
\begin{align*}& r_{12}\circ_D r_{13}+[r_{23},r_{13}]_D +r_{12}\circ_D r_{23}
\\=&\sum_{i,j=1}^{n}e_{i}\circ_D e_{j}\otimes e_{i}^{*}\otimes e_{j}^{*}+e_{j}\otimes e_{i} \otimes [e_{i}^{*},e_{j}^{*}]_D 
+e_{i}\otimes e_{i}^{*}\circ_D e_j\otimes e_{j}^{*}
\\=&\sum_{i,j=1}^{n}e_{i}\circ_A e_{j}\otimes e_{i}^{*}\otimes e_{j}^{*}+e_{j}\otimes e_{i} \otimes [e_{i}^{*},e_{j}^{*}]_{A^*}
+e_{i} \otimes [(R_{\circ_{A^*}}^{*}-L_{\circ_{A^*}}^{*})(e_{i}^{*})e_{j}+R_{\circ}^{*}(e_{j})e_{i}^{*}]\otimes e_{j}^{*}
\\=&0.\end{align*}\end{small}
By Eqs.~(\ref{Db1}) and (\ref{Db2}), we have for all $x\in A$,
\begin{small}
\begin{align*}&(L_{\circ_D}(x)\otimes I-I\otimes \mathrm{ad}_{D}(x))\sum_{i,j=1}^{n}(e_{i}\otimes e_{i}^{*}+e_{i}^{*}\otimes e_{i})
\\=&\sum_{i,j=1}^{n}x\circ_{D}e_{i}\otimes e_{i}^{*}-e_{i}\otimes
[x,e_{i}^{*}]_{D}+x\circ_{D}e_{j}^{*}\otimes e_{j}-e_{j}^{*}\otimes [x,e_{j}]_{D}
\\=&\sum_{i,j=1}^{n}x\circ_{A} e_{i}\otimes e_{i}^{*}-e_{i}\otimes (L_{\circ_{A^*}}^{*}(e_{i}^{*})x-L_{\circ_{A}}^{*}(x)e_{i}^{*})
+(R_{\circ_{A^*}}^{*}(e_{j}^{*})x-\mathrm{ad}_{A}^{*}(x)e_{j}^{*})\otimes e_{j}-e_{j}^{*}\otimes [x,e_{j}]_{A},
\end{align*}\end{small}
where $\mathrm{ad}_{A^*}^{*}=L_{\circ_{A^*}}^{*}-R_{\circ_{A^*}}^{*},~\mathrm{ad}_{A}^{*}=L_{\circ_{A}}^{*}-R_{\circ_{A}}^{*}$
and $[x,y ]_A=x\circ_A y-y\circ_A x,~[x,\zeta ]_D=x\circ_D \zeta-\zeta\circ_D x$.
Note that
\begin{small}
\begin{align*}& \sum_{i}^{n}x\circ_{A} e_{i}\otimes e_{i}^{*}+e_{i}\otimes L_{\circ_{A}}^{*}(x)e_{i}^{*}=0,
\\&\sum_{i=1}^{n}(R_{\circ_{A^*}}^{*}(e_{j}^{*})x\otimes e_{j}-e_{i}\otimes L_{\circ_{A^*}}^{*}(e_{i}^{*})x=0,
\\&\sum_{i}^{n}\mathrm{ad}_{A}^{*}(x)e_{j}^{*})\otimes e_{j}+e_{j}^{*}\otimes [x,e_{j}]_{A}=0.
\end{align*}\end{small}
Thus, $(L_{\circ_D}(x)\otimes I-I\otimes \mathrm{ad}_{D}(x))(r+\tau(r))=0.$
By duality, we get
\begin{equation*}(L_{\circ_D}(\zeta)\otimes I-I\otimes \mathrm{ad}_{D}(\zeta))(r+\tau(r))=0.\end{equation*}
for all $\zeta\in A^{*}$. In the light of Theorem 4.5 \cite{19}, 
\begin{align*}&(I\otimes L_{\star_D}(x)+L_{\ast_D}(x)\otimes I)(r+\tau(r))=0, \\
&(I\otimes L_{\star_D}(\zeta)+L_{\ast_D}(\zeta)\otimes I)(r+\tau(r))=0,\\
&r_{12}\ast_D r_{23}-r_{12}\ast_D r_{13}+r_{13}\star_D r_{23}=0.
\end{align*}
Thus, $r+\tau(r)$ is invariant and $r$ is a solution of the APP-YBE in $(D,\ast_D,\circ_D)$.
 Combining Theorem \ref{Qs}, $(D,\ast_D,\circ_D, \Delta_{r},\delta_{r})$ is a quasi-triangular anti-pre-Poisson bialgebra.
 Furthermore, the linear maps $T_{r},T_{\tau(r)}:D^{*}\longrightarrow D$ are respectively defined by 
$T_{r}(\zeta,x)=\zeta,~T_{\tau(r)}(\zeta,x)=-x$ for all $x\in A,\zeta\in A^{*}$. Thus,
$T_{r+\tau(r)}(\zeta,x)=(\zeta,-x)$ is a linear isomorphism. Hence, $(D,\ast_D,\circ_D, \Delta_{r},\delta_{r})$
is a factorizable anti-pre-Poisson bialgebra.
\end{proof}
%%%%%%%%%%%%%%%%%%%%%%%%%%%%%%%%%%%%%%%%%%%%%%%%%%%%%%%%%%%%%%%%%%%%%%%%%%%%%%%%%%%%%%%%%%%%%%%%%%%%%

%%%%%%%%%%%%%%%%%%%%%%%%%%%%%%%%%%%%%%%%%%%%%%%%%%%%%%%%%%%%%%%%%%%%%%%%%%%%%%%%%%%%%%%%%%%%%%%%%%%%%%%%%%
\section{Relative Rota-Baxter operators and quadratic Rota-Baxter anti-pre-Poisson algebras }
\subsection{Relative Rota-Baxter operators and the APP-YBE}
This section examines solutions of the APP-YBE with invariant symmetric parts 
through relative Rota-Baxter operators of weights on anti-pre-Poisson algebras.

\begin{defi}\label{RRa} Let $(A,\ast_A,\circ_A)$ and $(V,\ast_V,\circ_V)$ be anti-pre-Poisson algebras. Assume that
 $(V,l_{\ast},r_{\ast},$ \ \ \ $l_{\circ},r_{\circ})$ is a representation and the following conditions are satisfied:
 \begin{small}
 \begin{align}&\label{Ap1}
 l_{\ast}(x)(a\ast_{V}b)=- (l_{\star}(x)a)\ast_{V}b=-r_{\ast}(x)(a\star_{V}b)=b\ast_{V}(r_{\ast}(x)a),\\&
\label{Ap2}a\ast_{V}(l_{\ast}(x)b)=- (l_{\star}(x)a)\ast_{V}b=- (l_{\star}(x)b)\ast_{V}a=b\ast_{V}(l_{\ast}(x)a),
\\&
\label{Ap3}a\ast_{V}(r_{\ast}(x)b)=- r_{\ast}(x)(a\star_{V}b)=- (l_{\star}(x)b)\ast_{V}a= l_{\ast}(x)(b\ast_{V}a),\\&
\label{Ap4}b\ast_{V}(l_{\ast}(x)a)=l_{\ast}(x)(b\ast_{V}a),\\&
\label{Ap5}b\ast_{V}(r_{\ast}(x)a)=a\ast_{V}(r_{\ast}(x)b),\\&
\label{Ap6} l_{\circ}(x)(a\circ_{V}b)-a\circ_{V}(l_{\circ}(x)b)=(r_{\circ}(x)a-l_{\circ}(x)a)\circ_{V}b,\\&
\label{Ap7}(l_{\circ}(x)a-r_{\circ}(x)a)\circ_{V}b+r_{\circ}(x)[a,b]_{V}=(l_{\circ}(x)b-r_{\circ}(x)b)\circ_{V}a,\\&
\label{Ap8}a\circ_{V}(r_{\circ}(x)b)-b\circ_{V}(r_{\circ}(x)a)=r_{\circ}(x)[b,a]_{V},\\&
\label{Ap9}(l_{\circ}(x)a-r_{\circ}(x)a)\ast_{V}b=a\ast_{V}(l_{\circ}(x)b)-l_{\circ}(x)(a\ast_{V}b),\\&
\label{Ap10}(l_{\star}(x)a)\circ_{V}b=-l_{\ast}(x)(a\circ_Vb)-a\ast_{V}(l_{\circ}(x)b),\\&
\label{Ap11}(r_{\circ}(x)a-l_{\circ}(x)a)\ast_{V}b=l_{\ast}(x)(a\circ_{V}b)-a\circ_{V}(l_{\ast}(x)b),\\&
\label{Ap12}r_{\ast}(x)(a\circ_{V}b-b\circ_{V}a)=b\ast_{V}(r_{\circ}(x)a)-a\circ_{V}(r_{\ast}(x)b),\\&
\label{Ap13}r_{\circ}(x)(a\star_{V}b)=-a\ast_{V}(r_{\circ}(x)b)-b\ast_{V}(r_{\circ}(x)a),\\&
\label{Ap14}l_{\circ}(x)(a\star_{V}b)+a\ast_{V}(r_{\circ}(x)b)+b\ast_{V}(r_{\circ}(x)a)=a\ast_{V}(l_{\circ}(x)b)+b\ast_{V}(l_{\circ}(x)a),\\&
\label{Ap15}b\circ_{V}(l_{\star}(x)a)+l_{\ast}(x)(a\circ_Vb)+a\ast_{V}(l_{\circ}(x)b)=l_{\ast}(x)(b\circ_{V}a)+a\ast_{V}(r_{\circ}(x)b),
\end{align}\end{small}
 for all $x\in A$ and $a,b\in V$. Then $(V,\ast_V,\circ_V,l_{\ast},r_{\ast},l_{\circ},r_{\circ})$ is called
  an A-anti-pre-Poisson algebra, where $[a,b]_{V}=a\circ_Vb-b\circ_Va,~a\star_Vb=a\ast_Vb+b\ast_Va$ and
$l_{\star}=l_{\ast}+r_{\ast}$.
\end{defi}

By Proposition \ref{Mp}, $(V,\ast_V,\circ_V,l_{\ast},r_{\ast},l_{\circ},r_{\circ})$ is an A-anti-pre-Poisson algebra
 if and only if $(A\oplus V,\ast,\circ)$ is an anti-pre-Poisson algebra, where
 \begin{align*}&
 (x+a)\ast(y+b)=x\ast_A y+l_{\ast}(x)b+r_{\ast}(y)a+a\ast_V b,\\&
 (x+a)\circ (y+b)=x\circ_A y+l_{\circ}(x)b+r_{\circ}(y)a+a\circ_V b.
  \end{align*}
    
 \begin{pro}
Let $(V,\ast_V,\circ_V,l_{\ast},r_{\ast},l_{\circ},r_{\circ})$ be
  an A-anti-pre-Poisson algebra. Then $(V,\star_V,[ \ , \ ]_V,l_{\ast}+r_{\ast},l_{\circ}-r_{\circ})$ is
  an A-Poisson algebra.
 \end{pro}
 
\begin{proof}
It can be checked by direct computations.
\end{proof}

One can turn to \cite{17} for more information on A-Poisson algebra. 

 \begin{pro}
Let $(A,\ast,\circ)$ be an anti-pre-Poisson algebra and $r\in A\otimes A$ be symmetric and invariant. Define the
multiplications $\ast_r,\circ_r:A^{*}\otimes A^{*}\longrightarrow A^{*}$ by
\begin{align}&\label{Apa1}\zeta\ast_{r}\eta=L_{\star}^{*}(T_{r}(\zeta))\eta=-R_{\ast}^{*}(T_{r}(\eta))\zeta,
\\&\label{Apa2}\zeta\circ_{r}\eta=-\mathrm{ad}^{*}(T_{r}(\zeta))\eta=R_{\circ}^{*}(T_{r}(\eta))\zeta,~~\forall~\zeta,\eta\in A^{*}.
\end{align} 
Then $(A^{*},\ast_r,\circ_r,L_{\star}^{*}, -R_{\ast}^{*}, -\mathrm{ad}^{*},R_{\circ}^{*})$ is an A-anti-pre-Poisson algebra
and $(A^{*},\star_r,[ \ , \ ]_{r},L_{\ast}^{*},-L_{\circ}^{*})$ is an A-Poisson algebra, where
the Poisson algebra structure $(\star_r,[ \ , \ ]_{r})$ on $A^{*}$ is
given by 
\begin{small}\begin{align}&\label{Apa3}\zeta\star_{r}\eta=L_{\star}^{*}(T_{r}(\zeta))\eta
+L_{\star}^{*}(T_{r}(\eta))\zeta=-R_{\ast}^{*}(T_{r}(\eta))\zeta-
R_{\ast}^{*}(T_{r}(\zeta))\eta
,\\&\label{Apa4}
[ \zeta , \eta ]_{r}=-\mathrm{ad}^{*}(T_{r}(\zeta))\eta+\mathrm{ad}^{*}(T_{r}(\eta))\zeta=R_{\circ}^{*}(T_{r}(\eta))\zeta
-R_{\circ}^{*}(T_{r)}(\zeta))\eta.
\end{align}\end{small}
\end{pro}

\begin{proof} We only prove that Eq.~(\ref{Ap15}) holds in the case that 
$l_{\ast}=L_{\star}^{*}, r_{\ast}=-R_{\ast}^{*}, l_{\circ}=-\mathrm{ad}^{*},r_{\circ}=R_{\circ}^{*}$. 
The other cases follow analogously. Indeed, using Eqs.~(\ref{Ap1})-(\ref{Ap2}), (\ref{IE8}) and (\ref{r8}), we have
\begin{small}
\begin{align*}&b\circ_{V}(l_{\star}(x)a)+l_{\ast}(x)(a\circ_Vb)+a\ast_{V}(l_{\circ}(x)b)-l_{\ast}(x)(b\circ_{V}a)-a\ast_{V}(r_{\circ}(x)b)
\\=&L_{\star}^{*}(x)(a\circ_Vb-b\circ_{V}a)+b\circ_{V}L_{\ast}^{*}(x)(a)-a\ast_{V}(L_{\circ}^{*}(x)b)
\\=&L_{\star}^{*}(x)(R_{\circ}^{*}(T_{r}(b))a+\mathrm{ad}^{*}(T_{r}(b))a)
-\mathrm{ad}^{*}(T_{r}(b))L_{\ast}^{*}(x)(a)+R_{\ast}^{*}(T_{r}(L_{\circ}^{*}(x)b)a
\\=&L_{\star}^{*}(x)L_{\circ}^{*}(T_{r}(b))-\mathrm{ad}^{*}(T_{r}(b))L_{\ast}^{*}(x)-R_{\ast}^{*}([x,T_{r}(b)]_A )
\\=&0.
\end{align*}\end{small} 
This completes the proof.
\end{proof}

\begin{defi} Let $(A,\ast,\circ)$ be an anti-pre-Poisson algebra and $(V,\ast_V,\circ_V,l_{\ast},r_{\ast},l_{\circ},r_{\circ})$
 be an A-anti-pre-Poisson algebra.
A relative Rota-Baxter operator $T$ of weight $\lambda$ on $(A,\ast,\circ)$ associated to
 $(V,\ast_V,\circ_V,l_{\ast},r_{\ast},l_{\circ},r_{\circ})$
   is a linear map $T:V\longrightarrow A$ satisfying
\begin{align*}&T(u)\ast T(v)=T (l_{\ast}(T(u))v+r_{\ast}(T(v))u+\lambda u\ast_V v),\\& 
T(u)\circ T(v)=T (l_{\circ}(T(u))v+r_{\circ}(T(v))u+\lambda u\circ_V v),~~\forall~u,v\in V.\end{align*}
When $u\circ_Vv=u\ast_Vv=0$ for all $u,v\in V$, then $T$ is simply a relative Rota-Baxter operator ( $\mathcal O$-operator) on
$(A,\ast,\circ)$ associated to the representation $(V,l_{\ast},r_{\ast},l_{\circ},r_{\circ})$.
\end{defi}

\begin{thm}
Let $(A,\ast,\circ)$ be an anti-pre-Poisson algebra and $r\in A\otimes A$. Assume that $r+\tau(r)$ is invariant.
Then the following
conditions are equivalent.
 \begin{enumerate}
\item $r$ is a solution of the APP-YBE in $(A,\ast,\circ)$
 such that $(A,\ast,\circ,\Delta_{r},\delta_{r})$ is a quasi-triangular anti-pre-Poisson bialgebra
  with $\Delta_{r},\delta_{r}$ given by Eqs.~(\ref{CB1}) and (\ref{CB2}).
\item $T_r$ is a relative Rota–Baxter operator of weight $-1$ on $(A,\ast,\circ)$
 with respect to the A-anti-pre-Poisson algebra 
 $(A^{*},\ast_r,\circ_r,L_{\star}^{*}, -R_{\ast}^{*}, -\mathrm{ad}^{*},R_{\circ}^{*})$,
 that is,
 \begin{small}
 \begin{align}\label{AD5}&T_{r}(\zeta)\ast T_{r}(\eta)=T_{r}(L_{\star}^*(T_{r}(\zeta))\eta
 -R_{\ast}^*(T_{r}(\eta))\zeta-\zeta\ast_{r+\tau(r)}\eta), \\&
 \label{AD6}T_{r}(\zeta)\circ T_{r}(\eta)=T_{r}(-\mathrm{ad}^{*}(T_{r}(\zeta))\eta+R_{\circ}^*(T_{r}(\eta))\zeta
-\zeta\circ_{r+\tau(r)}\eta). 
 \end{align}\end{small}
\item $T_r$ is a relative Rota–Baxter operator of weight $-1$ on $(A,\star, [ \ , \ ])$
 with respect to the A-Poisson algebra $(A^{*},\star_r,[ \ , \ ]_{r},L_{\ast}^{*},-L_{\circ}^{*})$, that is,
 \begin{small}
 \begin{align}\label{AD8} &
 T_{r}(\zeta)\star T_{r}(\eta)=T_{r}(L_{\ast}^*(T_{r}(\zeta))\eta-L_{\circ}^*(T_{r}(\eta))\zeta-\zeta\star_{r+\tau(r)}\eta),\\&
 \label{AD9} [T_{r}(\zeta),T_{r}(\eta)]=T_{r}(-L_{\circ}^*(T_{r}(\zeta))\eta+L_{\circ}^*(T_{r}(\eta))\zeta-[\zeta,\eta]_{r+\tau(r)}), 
 \end{align}\end{small}
  \end{enumerate}
for all $\zeta,\eta\in A^{*}$, where $\ast_{r+\tau(r)},~\circ_{r+\tau(r)}$ and $\star_{r+\tau(r)},
[ \ , \ ]_{r+\tau(r)}$ are given by Eqs.~(\ref{Apa1})-(\ref{Apa4})
\end{thm}
  
 \begin{proof}
 Based on Proposition \ref{Ys1}, Theorem \ref{Ys3} and Theorem \ref{Ys4}, if $r+\tau(r)$ is invariant, then
 $r$ is a solution of the APP-YBE in $(A,\ast,\circ)$ if and only if
 \begin{align*}&T_{r}(\zeta)\circ T_{r}(\eta)=T_{r}(-\mathrm{ad}^{*}(T_{r}(\zeta))\eta-R_{\circ}^*(T_{\tau(r)}(\eta))\zeta),\\&
 T_{r}(\zeta)\ast T_{r}(\eta)=T_{r}(L_{\star}^{*}(T_{r}(\zeta))\eta-R_{\ast}^*(T_{\tau(r)}(\eta))\zeta) .\end{align*}
 Observe that
 \begin{small}
\begin{align*}T_{r}(\zeta)\circ T_{r}(\eta)&=T_{r}(-\mathrm{ad}^{*}(T_{r}(\zeta))\eta-R_{\circ}^*(T_{\tau(r)}(\eta))\zeta)
\\&=T_{r}(-\mathrm{ad}^{*}(T_{r}(\zeta))\eta+R_{\circ}^*(T_{r}(\eta))\zeta-R_{\circ}^*(T_{\tau(r)+r}(\eta))\zeta)
\\&=T_{r}(-\mathrm{ad}^{*}(T_{r}(\zeta))\eta+R_{\circ}^*(T_{r}(\eta))\zeta-\zeta\circ_{r+\tau(r)}\eta),\end{align*}
 \begin{align*}
 T_{r}(\zeta)\ast T_{r}(\eta)&=T_{r}(L_{\star}^*(T_{r}(\zeta))\eta
 +R_{\ast}^*(T_{\tau(r)}(\eta))\zeta-\zeta\ast_{r+\tau(r)}\eta)
\\&=T_{r}(L_{\star}^*(T_{r}(\zeta))\eta-R_{\ast}^*(T_{r}(\eta))\zeta
 +R_{\ast}^*(T_{r+\tau(r)}(\eta))\zeta)
 \\&=T_{r}(L_{\star}^*(T_{r}(\zeta))\eta-R_{\ast}^*(T_{r}(\eta))\zeta
 -\zeta\ast_{r+\tau(r)}\eta).\end{align*}\end{small}
Thus, Item (a) $\Longleftrightarrow$ Item (b).
Similarly, Item (a) $\Longleftrightarrow$ Item (c).
  \end{proof}
 If $r\in A\otimes A$ is skew-symmetric, then $r+\tau(r)=0$. Thus, Theorem \ref{Op1} is obtained.
 
%%%%%%%%%%%%%%%%%%%%%%%%%%%%%%%%%%%%%%%%%%%%%%%%%%%%%%%%%%%%%%%%%%%%%%%%%%%%%%%%%%%%%%%%%%%%%%%%%%%%%%%%%%
\subsection{Quadratic Rota-Baxter anti-pre-Poisson algebras and factorizable anti-pre-Poisson bialgebras}
In this section, we first introduce the notion of quadratic Rota-Baxter anti-pre-Poisson algebras. 
We then establish the relationship between factorizable anti-pre-Poisson bialgebras 
and these quadratic Rota-Baxter anti-pre-Poisson algebras.

\begin{defi} Let $(A,\ast,\circ,P)$ be a Rota-Baxter anti-pre-Poisson algebra of weight $\lambda$
and $(A,\ast,\circ,\omega)$ a quadratic anti-pre-Poisson algebra. Then $(A,\ast,\circ,P,\omega)$
is called a \textbf{quadratic Rota-Baxter anti-pre-Poisson algebra of weight $\lambda$} if the following condition holds:
\begin{equation} \label{Fs}\omega (P(x),y)+\omega(x, P(y))+\lambda\omega(x,y)=0, ~\forall~x, y \in A.\end{equation}
\end{defi}

\begin{defi} Let $(A,\star,[ \ , \ ],P)$ be a Rota-Baxter Poisson algebra of weight $\lambda$ and
 $(A,\star,[ \ , \ ],\omega )$ be a quadratic Poisson algebra if the following condition holds: 
 \begin{equation} \label{Fs1}\omega (P(x),y)+\omega(x, P(y))+\lambda\omega(x,y)=0, ~\forall~x, y \in A.\end{equation}
Then $(A,\star,[ \ , \ ],P,\omega)$ is called a \textbf{quadratic Rota-Baxter Poisson algebra of 
weight $\lambda$}.
\end{defi}

Observe that for all $x, y \in A$,
\begin{align*} &\lambda\omega(x,y)+\omega (-\lambda(x)- P(x),y)+\omega(x, -\lambda(y)- P(y))\\=&
-\lambda\omega(x,y)-\omega (P(x),y)-\omega(x,  P(y)), \end{align*}
from which we obtain the following conclusions.

\begin{pro} \label{Fb2} Let $(A,\ast,\circ,\omega)$ be a quadratic anti-pre-Poisson algebra and let $P:A\longrightarrow A$ be a linear
map. Then $(A,\ast,\circ,P,\omega)$ is a quadratic Rota–Baxter anti-pre-Poisson algebra of weight $\lambda$ if and only if
$(A,\ast,\circ,-\lambda I-P,\omega)$ is a quadratic Rota–Baxter anti-pre-Poisson algebra of weight $\lambda$.
\end{pro}

\begin{pro} Let $(A,\star,[ \ , \ ],\omega)$ be a quadratic Poisson algebra and let $P:A\longrightarrow A$ be a linear
map. Then $(A,\star,[ \ , \ ] ,P,\omega)$ is a quadratic Rota-Baxter Poisson algebra of weight $\lambda$ if and only if
$(A,\star,[ \ , \ ],-\lambda I-P,\omega)$ is a quadratic Rota-Baxter Poisson algebra of weight $\lambda$.
\end{pro}

\begin{thm}\label{Fb0} Suppose that $(A,\star,[ \ , \ ],P,\omega)$ is a quadratic Rota-Baxter Poisson algebra of 
weight $\lambda$. Then $(A,\ast,\circ,P,\omega)$ is a quadratic Rota-Baxter anti-pre-Poisson
algebra of weight $\lambda$, where $\ast,\circ$ are defined by Eqs.~(\ref{C1}) and (\ref{C2}) respectively.
 On the other hand, let $(A,\ast,\circ,P,\omega)$ be a quadratic Rota-Baxter anti-pre-Poisson algebra of weight $\lambda$.
Then $(A,\star,[ \ , \ ],P,\omega)$ is a quadratic Rota-Baxter Poisson algebra
 of weight $\lambda$, where $x\star y=x\ast y+y\ast x,~[x,y]=x\circ y-y\circ x$.
\end{thm}

\begin{proof} Let $(A,\star,[ \ , \ ],P,\omega)$ is a quadratic Rota-Baxter Poisson algebra of 
weight $\lambda$. In the light of Proposition \ref{Qc}, $(A,\ast,\circ,\omega)$ is a quadratic anti-pre-Poisson algebra.
Using Eqs.~(\ref{C2}) and (\ref{Fs1}), we obtain for all $x,y,z\in A$, 
\begin{small}
\begin{align*}& \omega( P(x)\circ P(y)-P(P(x)\circ y+x\circ P(y)+\lambda x\circ y),z)
\\=&\omega( P(y), [P(x),z])+\omega(P(x)\circ y+x\circ P(y)+\lambda x\circ y, P(z))
+\lambda\omega(P(x)\circ y+x\circ P(y)+\lambda x\circ y,z)
\\=&\omega( P(y), [P(x),z])+\omega(y, [P(x),P(z)])+\omega(  P(y),[x, P(z)])+\lambda\omega(  y, [x,P(z)])
\\&+\lambda\omega(y,[P(x),z])+\lambda\omega(P(y),[x,z])+\lambda^{2} \omega( y,[x,z])
\\=&\omega(y, [P(x),P(z)])+\omega(P(y),[x, P(z)]+ [P(x),z]+\lambda [x,z])
+\lambda\omega(y,[x, P(z)]+ [P(x),z]+\lambda [x,z])
\\=&\omega(y, [P(x),P(z)])-\omega(y,P([x, P(z)]+ [P(x),z]+\lambda [x,z]))
\\=&0.\end{align*}\end{small}
It follows that $P(x)\circ P(y)=P(P(x)\circ y+x\circ P(y)+\lambda x\circ y)$.
By Theorem 6.7 \cite{19}, $P(x)\ast P(y)=P(P(x)\ast y+x\ast P(y)+\lambda x\ast y)$.
Thus, $(A,\ast,\circ,P,\omega)$ is a quadratic Rota-Baxter anti-pre-Poisson algebra of weight $\lambda$.
The other hand is apparently.
We complete the proof.
\end{proof}

Let $\omega$ be a non-degenerate bilinear form on a vector space $A$. Then there is an isomorphism
$\omega^{\sharp}:A\longrightarrow A^{*}$ given by
\begin{equation} \omega(x,y)=\langle\omega^{\sharp}(x),y \rangle,~~\forall~x,y\in A.\end{equation}
Define an element $r_{\omega}\in A\otimes A$ such that $T_{r_{\omega}}=(\omega^{\sharp})^{-1}$,
that is, 
\begin{equation}\label{Nd1} \langle T_{r_{\omega}}(\zeta),\eta\rangle=\langle r_{\omega},\zeta \otimes\eta\rangle=
\langle (\omega^{\sharp})^{-1}(\zeta), \eta\rangle, \ \ \forall~\zeta,\eta\in A^{*}. \end{equation}

\begin{lem}\label{Fb1} Let $(A, \ast,\circ)$ be an anti-pre-Poisson algebra and $\omega$ be a non-degenerate bilinear form on $A$. Then
$(A, \ast,\circ,\omega)$ is a quadratic anti-pre-Poisson algebra if and only if the corresponding $r_{\omega}\in A\otimes A$ given by
Eq.~(\ref{Nd1}) is symmetric and invariant.\end{lem}
\begin{proof}
It is obvious that $\omega$ is symmetric if and only if $r_{\omega}$ is symmetric. 
For all $x,y\in A$, put $\omega^{\sharp}(x)=\zeta,\omega^{\sharp}(y)=\eta,
\omega^{\sharp}(z)=\theta$ with $\zeta,\eta,\theta\in A^{*}$. If $\omega$ is symmetric, we have
\begin{small}
\begin{align*} \omega (x \circ y, z)-\omega(y, [x, z])
=&\langle \omega^{\sharp}(z), x\circ(\omega^{\sharp})^{-1}(\eta)\rangle
-\langle \omega^{\sharp}(y), [x, (\omega^{\sharp})^{-1}(\theta)]\rangle
\\=&\langle \theta, x\circ T_{r_{\omega}}(\zeta)\rangle
-\langle \eta, [x, T_{r_{\omega}}(\theta)]\rangle
\\=&\langle -R_{\circ}^{*}(T_{r_{\omega}}(\eta))\theta-\mathrm{ad}^{*}(T_{r_{\omega}}(\theta))\eta, x\rangle.\end{align*}
\end{small}
Thus, if $\omega$ is symmetric, $\omega (x \circ y, z)=\omega(y, [x, z])$ if and only if
$R_{\circ}^{*}(T_{r_{\omega}}(\eta))\theta=-\mathrm{ad}^{*}(T_{r_{\omega}}(\theta))\eta$.
By Lemma 6.8 \cite{19}, if $\omega$ is symmetric, $\omega (x \ast y, z)=-\omega(x, y\star z)$ if and only if
$ L_{\star}^{*} (T_{r_{\omega}}(\zeta))\eta=-R_{\ast}^{*}(T_{r_{\omega}}(\eta))\zeta).$
Combining Proposition \ref{Si}, we complete the proof.
 \end{proof}
 
\begin{pro} \label{QF1} Let $(A, \ast,\circ,\omega)$ be a quadratic anti-pre-Poisson algebra and $r\in A\otimes A$. 
Assume that $r+\tau(r)$ is invariant. Define a linear map 
\begin{equation}\label{Nd2}P:A\longrightarrow A,\ \ \ P(x)=T_{r}\omega^{\sharp}(x),~~\forall~x\in A.\end{equation}
Then $r$ is a solution of the APP-YBE in $(A, \ast,\circ)$ if and only if $P$ satisfies 
\begin{small}
\begin{align}&\label{Nd3}P(x)\circ P(y)= P(P(x)\circ y+x\circ P(y)-x\circ T_{r+\tau(r)}\omega^{\sharp}(y)),
\\&\label{Nd4}P(x)\ast P(y)= P(P(x)\ast y+x\ast P(y)-x\ast T_{r+\tau(r)}\omega^{\sharp}(y)),~~\forall~x,y\in A.
 \end{align}\end{small}
\end{pro}

\begin{proof} By Lemma \ref{Fb1}, $r_{\omega}$ is symmetric and invariant. In the light of Proposition \ref{Si},
 Eqs.~(\ref{LR1}), (\ref{LR2}) and (\ref{Apa2}),
 for all $x,y\in A$, put $\omega^{\sharp}(x)=\zeta,\omega^{\sharp}(y)=\eta$ with 
$\zeta,\eta\in A^{*}$, we get
\begin{small}
\begin{align*}& P(x)\circ P(y)=T_{r}\omega^{\sharp}(x)\circ T_{r}\omega^{\sharp}(y)=T_{r}(\zeta)\circ T_{r}(\eta),\\
&P(P(x)\circ y)=T_{r}\omega^{\sharp}(T_{r}(\zeta)\circ T_{r_{\omega}}(\eta))=
-T_{r}\omega^{\sharp} T_{r_{\omega}}(\mathrm{ad}^{*}(T_{r}(\zeta))\eta)
=-T_{r}(\mathrm{ad}^{*}(T_{r}(\zeta))\eta),
\\
&P(x\circ P(y))=T_{r}\omega^{\sharp}(T_{r_{\omega}}(\zeta)\circ T_{r}(\eta))=
T_{r}\omega^{\sharp}T_{r_{\omega}}(R_{\circ}^{*}(T_{r}(\eta))\zeta)
=T_{r}(R_{\circ}^{*}(T_{r}(\eta))\zeta),
\\
&P(x\circ T_{r+\tau(r)}\omega^{\sharp}(y))
=T_{r}\omega^{\sharp}(T_{r_{\omega}}(\zeta)\circ T_{r+\tau(r)}(\eta))
=T_{r}\omega^{\sharp}  T_{r_{\omega}}( R_{\circ}^{*}(T_{r+\tau(r)}(\eta))\zeta)
\\=&T_{r}(( R_{\circ}^{*}(T_{r+\tau(r)}(\eta))\zeta)
=T_{r}(\zeta\circ_{r+\tau(r)}\eta).\end{align*}\end{small}
 Thus, Eq.~(\ref{Nd3}) holds if and only if Eq.~(\ref{AD6}) holds.
By the same token, we can verify that Eq.~(\ref{Nd4}) holds if and only if Eq.~(\ref{AD5}) holds. The proof is completed.
 \end{proof}
 
\begin{lem} \label{QF2} Let $A$ be a vector space and $\omega$ be a non-degenerate symmetric bilinear form. 
Let $r\in A\otimes A$,
$\lambda\in k$ and $P$ given by Eq.~(\ref{Nd2}). Then $r$ satisfies
\begin{equation} \label{Nd5} r+\tau(r)=-\lambda r_{\omega} \end{equation}
if and only if $P$ satisfies Eq.~(\ref{Fs}).
\end{lem}
\begin{proof} For all $x,y\in A$, put $\omega^{\sharp}(x)=\zeta,\omega^{\sharp}(y)=\eta$ with $\zeta,\eta\in A^{*}$.
\begin{small}
\begin{align*}& \omega(P(x),y)=\omega(y,P(x))=\langle\omega^{\sharp}(y),T_{r}\omega^{\sharp}(x)\rangle
=\langle \eta,T_{r}(\zeta)\rangle
=\langle r,\zeta\otimes\eta\rangle,\\
&\omega(x,P(y))=\langle \omega^{\sharp}(x),T_{r}\omega^{\sharp}(y)\rangle
=\langle \zeta,T_{r}(\eta)\rangle
=\langle r,\eta\otimes\zeta\rangle
=\langle \tau(r),\zeta\otimes\eta\rangle,
\\
&\lambda\omega(x,y)=\lambda \omega(y,x)
=\lambda\langle \omega^{\sharp}(y),(\omega^{\sharp})^{-1}\omega^{\sharp}(x)\rangle
=\lambda\langle r_{\omega},\zeta\otimes\eta\rangle
.\end{align*}\end{small}
Thus, Eq.~(\ref{Nd5}) holds if and only if Eq.~(\ref{Fs}) holds.
 \end{proof}
 
\begin{cor} \label{Fb3} Let $(A, \ast,\circ,P,\omega)$ be a quadratic Rota-Baxter anti-pre-Poisson algebra of weight 0.
 Then there is a triangular
anti-pre-Poisson bialgebra $(A, \ast,\circ, \Delta_{r},\delta_{r})$ with $\Delta_{r},\delta_{r}$ given
 by Eqs.~(\ref{CB1}) and (\ref{CB2}), where $r\in A\otimes A$ given by 
$T_{r}(\zeta)=P(\omega^{\sharp})^{-1}(\zeta)$ for all $\zeta\in A^{*}$.
\end{cor}

 \begin{proof} Since $r+\tau(r)=-\lambda r_{\omega}=0$, $r$ is skew-symmetric.
On the basis of Proposition \ref{QF1} and Lemma \ref{QF2}, we get the conclusion. \end{proof}
 
\begin{thm} \label{Fb3} Let $(A, \ast,\circ, \Delta_{r},\delta_{r})$ be a factorizable 
anti-pre-Poisson bialgebra
with $r\in A\otimes A$. Then $(A, \ast,\circ,P,\omega)$
 is a quadratic Rota-Baxter anti-pre-Poisson algebra of weight $\lambda$ with $P$ given by Eq.~(\ref{Nd2}), and $\omega$ is given by
\begin{equation} \label{Nd6} \omega(x,y)=-\lambda\langle T_{r+\tau(r)}^{-1}(x),y \rangle,~~\forall~x,y\in A.\end{equation}
Conversely, let $(A,\ast,\circ,P,\omega)$ be a quadratic Rota-Baxter anti-pre-Poisson algebra of weight $\lambda~(\lambda\neq 0)$.
Then there is a factorizable anti-pre-Poisson bialgebra $(A, \ast,\circ, \Delta_{r},\delta_{r})$
 with $\Delta_{r},\delta_{r}$ defined by Eqs.~(\ref{CB1}) and (\ref{CB2}), where $r\in A\otimes A$ is
given through the operator form $T_r=P(\omega^{\sharp})^{-1}$.
\end{thm}

\begin{proof} On the one hand, since $(A, \ast,\circ, \Delta_{r},\delta_{r})$ is a factorizable 
anti-pre-Poisson bialgebra, $r+\tau(r)$ is invariant and $T_{r+\tau(r)}$ is a linear isomorphism. 
By Proposition \ref{QF1} and Lemma \ref{QF2}, we get that $(A,\ast,\circ,P,\omega)$
 is a quadratic Rota-Baxter anti-pre-Poisson algebra of weight $\lambda$, where $\omega^{\sharp}=-\lambda T_{r+\tau(r)}^{-1}$. 
 Conversely, assume that $(A,\ast,\circ,P,\omega)$ is a quadratic Rota-Baxter anti-pre-Poisson algebra of weight $\lambda~(\lambda\neq 0)$.
 In the light of Lemma \ref{Fb1},  Lemma \ref{QF2} and Proposition \ref{QF1}, 
 $r+\tau(r)$ is invariant, $T_{r+\tau(r)}=-\lambda (\omega^{\sharp})^{-1}$ is a linear isomorphism
 and $r$ is a solution of the APP-YBE in $(A, \ast,\circ)$. Thus,
  $(A, \ast,\circ, \Delta_{r},\delta_{r})$ is a factorizable anti-pre-Poisson bialgebra.
\end{proof}

By Theorem \ref{Fb0} and Theorem \ref{Fb3}, we have the following chain of correspondences:
\[(A, \star, [ \ , \ ], P, \omega){\overset{\text{Theorem 5.10}}{\longleftrightarrow}} 
(A, \ast,\circ, P, \omega){\overset{\text{Theorem 5.15}}{\longleftrightarrow}}(A, \ast,\circ, \Delta_r, \delta_r).\]

\begin{pro} Let $(A, \ast,\circ,P)$ be a Rota–Baxter anti-pre-Poisson algebra of weight $\lambda$. Then 
$(A\ltimes A^{*},\omega,P-(P^{*}+\lambda I))$
is a quadratic Rota–Baxter anti-pre-Poisson algebra of weight $\lambda$, where the bilinear form $\omega$ on $A\oplus A^{*}$
is given by
\begin{equation*} \omega(x+\zeta,y+\eta)=\langle x,\eta\rangle+\langle y,\zeta\rangle,~~\forall~x,y\in A,~\zeta,\eta\in A^{*}.\end{equation*}
Then $(A\ltimes A^{*}, \Delta_{r},\delta_{r})$
is a factorizable anti-pre-Poisson bialgebra
with $\Delta_{r},\delta_{r}$ defined by Eqs.~(\ref{CB1}) and (\ref{CB2}) with $r$ given by $T_r=P(\omega^{\sharp})^{-1}$.
  Explicitly, assume that $\{e_1,\cdot\cdot\cdot, e_n\}$
is a basis of $A$ and $\{e_{1}^{*},\cdot\cdot\cdot, e^{*}_n\}$
is the dual basis, where $r=\sum_{i}e_{i}^{*}\otimes P(e_{i})-(P+\lambda I)(e_{i})\otimes e_{i}^{*}$.
\end{pro}

\begin{proof} By direct computations, $(A\ltimes A^{*},\omega,P-(P^{*}+\lambda I))$
is a quadratic Rota-Baxter anti-pre-Poisson algebra of weight $\lambda$.
For all $x\in A$ and $\zeta\in A^{*}$, $\omega^{\sharp}(x+\zeta)=x+\zeta$. By Corollary \ref{Fb3}, there is
 a linear map $T_{r}:A\oplus A^{*}\longrightarrow A\oplus A^{*}$ defined by
\begin{equation*} 
T_{r}(x+\zeta)=(P-(P^{*}+\lambda I))(\omega^{\sharp})^{-1}(x+\zeta)=P(x)-(P^{*}+\lambda I)(\zeta).\end{equation*}
Thus, \begin{small}
\begin{align*}&
\sum_{i,j}\langle r,e_{i}\otimes e_{j}^{*}\rangle=\sum_{i,j}\langle T_{r}(e_{i}), e_{j}^{*}\rangle=\sum_{i,j}\langle P(e_{i}), e_{j}^{*}\rangle
\\&\sum_{i,j}\langle r,e_{i}^{*}\otimes e_{j}\rangle=\sum_{i,j}\langle T_{r}(e_{i}^{*}), e_{j}\rangle=-\sum_{i,j}\langle (P^{*}+\lambda I)(e_{i}^{*}), e_{j}\rangle=-\sum_{i,j}\langle e_{i}^{*}, (P+\lambda I)(e_{j})\rangle.\end{align*}\end{small}
It follows that $r=\sum_{i}e_{i}^{*}\otimes P(e_{i})-(P+\lambda I)(e_{i})\otimes e_{i}^{*}$.
\end{proof}

 Building upon Proposition~\ref{Fb2} and Theorem~\ref{Fb3}, note that a quadratic Rota-Baxter anti-pre-Poisson algebra 
 $(A, \ast, \circ, P, \omega)$ of non-zero weight corresponds to a factorizable anti-pre-Poisson bialgebra. 
 It follows that $(A, \ast, \circ, -\lambda I -P, \omega)$ corresponds to such a bialgebra. 
 Moreover, by Proposition~\ref{Qf} and Theorem~\ref{Fb3}, 
 if $(A, \ast, \circ, \Delta_{r}, \delta_{r})$ is a factorizable anti-pre-Poisson bialgebra, 
 then $(A, \ast, \circ, \Delta_{\tau(r)}, \delta_{\tau(r)})$ 
 also induces a quadratic Rota–Baxter anti-pre-Poisson algebra of weight $\lambda$. Indeed, we have

\begin{pro}
Let $(A, \ast,\circ, \Delta_{r},\delta_{r})$ be a factorizable 
anti-pre-Poisson bialgebra which corresponds to a quadratic Rota-Baxter anti-pre-Poisson algebras of non-zero weight $\lambda$. 
Then the factorizable 
anti-pre-Poisson bialgebra $(A, \ast,\circ, \Delta_{\tau(r)},\delta_{\tau(r)})$ corresponds to the 
quadratic Rota-Baxter anti-pre-Poisson algebra $(A,\ast,\circ,-\lambda I-P,\omega)$ of non-zero weight $\lambda$. 
	
\end{pro}

\begin{proof} According to
 Proposition \ref{Qf} and Theorem \ref{Fb1}, there is  a
quadratic Rota-Baxter anti-pre-Poisson algebra $(A,\ast,\circ,P^\prime,\omega^\prime)$ of non-zero weight $\lambda$
 corresponds to the factorizable anti-pre-Poisson bialgebra $(A, \ast,\circ, \Delta_{\tau(r)},\delta_{\tau(r)})$.
By Theorem \ref{Fb3},
\begin{equation*} 
\omega^{\prime}(x,y)=-\lambda\langle T_{r+\tau(r)}^{-1}(x),y \rangle=\omega(x,y).\end{equation*}
Using Eqs.~(\ref{Nd2}) and (\ref{Nd6}), 
\begin{small}
\begin{align*} 
P^{\prime}(x)= T_{\tau(r)}{\omega^{\prime}}^{\sharp}(x)=T_{\tau(r)}\omega^{\sharp}(x)
&=-\lambda T_{\tau(r)}T_{r+\tau(r)}^{-1}(x)=\lambda (T_{r}-T_{r+\tau(r)})T_{r+\tau(r)}^{-1}(x)
\\&=\lambda T_{r}T_{r+\tau(r)}^{-1}(x)-\lambda (x)
=-T_{s}\omega^{\sharp}(x)-\lambda (x)=-P(x)-\lambda (x)
.\end{align*}\end{small}
Thus, the factorizable 
anti-pre-Poisson bialgebra $(A, \ast,\circ, \Delta_{\tau(r)},\delta_{\tau(r)})$ yields a
quadratic Rota-Baxter anti-pre-Poisson algebra $(A,\ast,\circ,-\lambda I-P,\omega)$ of non-zero weight $\lambda$.
Analogously, the converse part holds.
\end{proof}

%%%%%%%%%%%%%%%%%%%%%%%%%%%%%%%%%%%%%%%%%%%%%%%%%%%%%%%%%%%%%%%%%%%%%%%%%%%%%%%%%%%%%%%%%%%%%%%%%%%%%%%%%%

\begin{center}{\textbf{Acknowledgments}}
\end{center}
This work was supported by the Natural Science
Foundation of Zhejiang Province of China (LY19A010001), the Science
and Technology Planning Project of Zhejiang Province
(2022C01118).

%%%%%%%%%%%%%%%%%%%%%%%%%%%%%%%%%%%%%%%%%%%%%%%%%%%%%%%%%%%%%%%%%%%%%%%%%%%%%%%%%%%%%%%%%%%%%%%%%%%%%%%%%%
\begin{center} {\textbf{Statements and Declarations}}
\end{center}
 All datasets underlying the conclusions of the paper are available
to readers. No conflict of interest exits in the submission of this
manuscript.

%%%%%%%%%%%%%%%%%%%%%%%%%%%%%%%%%%%%%%%%%%%%%%%%%%%%%%%%%%%%%%%%%%%%%%%%%%%%%%%%%%%%%%%%%%%%%%%%%%%%%%%%%%

\end {document}